\documentclass{article} 
\usepackage{nips14submit_e,times}
\usepackage{url}
\usepackage{booktabs}

\usepackage{times}
\usepackage{graphicx} 
\usepackage{subfigure} 
\usepackage{varwidth}

\usepackage{algorithm}
\usepackage{algorithmic}

\usepackage{amsfonts}
\usepackage{mathrsfs}
\usepackage{amsmath, amsthm}
\usepackage{amssymb,url,paralist}

\newtheorem{thm}{Theorem}
\newtheorem{prop}{Proposition}

\newtheorem{lemma}{Lemma}

\newtheorem{definition}[thm]{Definition}
\newtheorem{assumption}{Assumption}
\newtheorem{observation}{Observation}

\def \R {\mathbb{R}}
\def \D {\mathcal{D}}

\def \E {\mathrm{E}}
\def \x {\mathbf{x}}

\def \X {\mathcal{X}}

\def \xh {\widehat{\x}}

\def \u {\mathbf{u}}

\def \v {\mathbf{v}}
\def \w {\mathbf{w}}

\def \R {\mathbb{R}}

\def \Hh {\widehat{\H}}

\def \R {\mathbb{R}}
\def \x {\mathbf{x}}
\def \E {\mathrm{E}}

\def \u {\mathbf{u}}

\def \xh {\widehat{\x}}

\def \w {\mathbf{w}}

\def \v {\mathbf{v}}

\def \D {\mathcal{D}}

\def \X {\mathcal{X}}

\def \Hh {\widehat{H}}

\title{On Data Preconditioning for Regularized Loss Minimization}

\author{
Tianbao Yang$^1$\thanks{tianbao-yang@uiowa.edu},  Rong Jin$^{2,3}$,  Shenghuo Zhu$^3$, Qihang Lin$^4$\\
$^1$Department of Computer Science, the University of Iowa\\
$^2$Department of Computer Science and Engineering, Michigan State University\\
$^3$Alibaba Group \\
$^4$Department of Management Sciences, the University of Iowa\\
}

\nipsfinalcopy 

\begin{document}

\maketitle
\vspace*{-0.25in}
\begin{abstract}
In this work, we study data preconditioning, a well-known and long-existing technique, for boosting the convergence of first-order methods for regularized loss minimization in machine learning. It is well understood that the condition number of the problem, i.e., the ratio of the Lipschitz constant to the strong convexity modulus, has a harsh effect on the convergence of the first-order optimization methods. Therefore,  minimizing a small regularized loss for achieving good generalization performance, yielding an ill conditioned problem,  becomes the bottleneck for big data problems. We provide a theory on data preconditioning for regularized loss minimization. In particular, our analysis exhibits an appropriate data preconditioner that is similar to zero component analysis (ZCA) whitening. Exploiting the concepts of numerical rank and coherence, we  characterize the conditions on the loss function and on the data under which data preconditioning can reduce the condition number and therefore boost the convergence for minimizing the regularized loss. To make the data preconditioning practically useful, we propose an efficient preconditioning method through random sampling. The preliminary experiments on simulated data sets and real data sets validate our theory.
\end{abstract}

\section{Introduction}
\label{submission}
Many supervised machine learning tasks end up with solving the following regularized loss minimization (RLM) problem:  
\begin{align}\label{eqn:p}
\min_{\w\in\R^d} \frac{1}{n}\sum_{i=1}^n\ell(\x_i^{\top}\w, y_i) + \frac{\lambda}{2}\|\w\|_2^2,
\end{align}
where $\x_i\in\mathcal X\subseteq\R^d$ denotes the feature representation, $y_i\in\mathcal Y$ denotes the supervised information, $\w\in\R^d$ represents the decision vector  and $\ell(z,y)$ is a convex loss function with respect to $z$. Examples can be found in classification (e.g., $\ell(\x^{\top}\w, y) = \log(1+\exp(-y\x^{\top}\w))$ for logistic regression) and regression (e.g., $\ell(\x^{\top}\w, y) = (1/2)(\x^{\top}\w - y)^2$ for least square regression). 

The first-order methods that base on the first-order information (i.e., gradient) have recently become the dominant approaches for solving the optimization problem in~(\ref{eqn:p}), due to their light computation compared to the second-order methods (e.g., the Newton method).  Because of  the explosive growth of data, recently many stochastic optimization algorithms have emerged  to further  reduce the running time of full gradient methods~\cite{opac-b1104789}, including stochastic gradient descent (SGD)~\cite{DBLP:conf/icml/Shamir013,DBLP:journals/mp/Shalev-ShwartzSSC11}, stochastic average gradient (SAG)~\cite{DBLP:conf/nips/RouxSB12}, stochastic dual coordinate ascent (SDCA)~\cite{DBLP:journals/jmlr/Shalev-Shwartz013,Hsieh:2008:DCD:1390156.1390208}, stochastic variance reduced gradient (SVRG)~\cite{NIPS2013_4937}. 
One limitation of most first-order methods is that they suffer from a poor convergence if the condition number is small. For instance, the gradient-based stochastic optimization algorithm Pegasos~\cite{DBLP:journals/mp/Shalev-ShwartzSSC11} for solving Support Vector Machine (SVM) with a Lipschitz continuous loss function, has  a convergence rate of $O\left(\frac{\bar L^2}{\lambda T}\right)$, where $\bar L$ is the Lipschitz constant of the loss function w.r.t $\w$. The convergence rate reveals that the smaller the condition number (i.e., $\bar L^2/\lambda$), the worse the convergence. The same phenomenon occurs in optimizing a smooth loss function. Without loss of generality, the iteration complexity -- the number of iterations required for achieving an $\epsilon$-optimal solution,  of SDCA, SAG and SVRG for a $L$-smooth loss function (whose gradient is  $\bar L$-Lipschitz continuous) is $O((n + \frac{\bar L}{\lambda})\log(\frac{1}{\epsilon}))$. Although the convergence is linear for a smooth loss function, however, iteration complexity would be dominated by the condition number $\bar L/\lambda$ if it is substantially large~\footnote{The condition number of the problem in~(\ref{eqn:p}) for the Lipschitz continuous loss function is referred to $\bar L^2/\lambda$, and for the smooth loss function is referred to $\bar L/\lambda$, where $\bar L$ is the Lipschitz constant for the function and its gradient w.r.t $\w$, respectively. }. As supporting evidences, many studies have found that setting $\lambda$ to a very small value plays a pivotal role in achieving good generalization performance~\cite{conf/icml/Shalev-ShwartzS08,conf/nips/SridharanSS08}, especially for data sets with a large number of examples.  Moreover, some theoretical analysis indicates that the value of $\lambda$ could be as small as $1/n$ in order to achieve a small generalization error~\cite{conf/icml/Shalev-ShwartzS08,DBLP:journals/jmlr/Shalev-Shwartz013}. Therefore, it arises as  an interesting question ``\textit{can we design first-order optimization algorithms that have less severe  and even no dependence on the large condition number}"?

While most previous works target on improving the convergence rate by achieving a better dependence on  the  number of iterations $T$, few works have revolved around  mitigating  the dependence on the condition number.  \cite{DBLP:conf/nips/BachM13}  provided a new analysis of the averaged stochastic gradient (ASG) algorithm for minimizing a smooth objective function with a constant step size. They established  a convergence rate of $O(1/T)$ without suffering from the small strong convexity modulus (c.f. the definition given in {\bf Definition}~\ref{def:2}). Two recent  works~\cite{journals/corr/NeedellSW13,Xiao:arXiv1403.4699} proposed to use importance sampling instead of random sampling in stochastic gradient methods, leading to a dependence on the averaged Lipschitz constant of the individual loss functions instead of the worst Lipschitz constant. However, the convergence rate still badly depends on $1/\lambda$. 

In this paper, we explore the data preconditioning for reducing the condition number of the problem~(\ref{eqn:p}). In contrast to many other works, the proposed data preconditioning technique can be potentially applied together with any first-order methods to improve their convergences. Data preconditioning is a long-existing technique that was used to improve the condition  number of a data matrix. In the general form, data preconditioning is to apply $P^{-1}$ to the data, where $P$ is a non-singular matrix. It has been employed widely in solving linear systems~\cite{Axelsson:1995:ISM:206744}.  In the context of convex optimization, data preconditioning has been applied to conjugate gradient and Newton methods  to improve their convergence for ill-conditioned problems~\cite{langer2007preconditioned}. However,  it remains unclear how  data preconditioning can be used to  improve the convergence of first-order methods for minimizing a regularized empirical  loss. In the context of non-convex optimization, the data preconditioning by ZCA whitening has been widely adopted  in learning deep neural networks from image data to speed-up the optimization~\cite{journals/jmlr/RanzatoKH10,lecun-efficient-backprop-1998}, though the underlying theory is barely known. Interestingly, our analysis reveals that the proposed data preconditioner is closely related to ZCA whitening and therefore shed light on the practice widely deployed in deep learning.  However, an inevitable critique on the usage of data preconditioning is the computational overhead pertaining to computing the preconditioned data. Thanks to modern cluster of computers, this computational overhead can be made as minimal as possible with parallel computations  (c.f. the discussions in subsection~\ref{sec:imp}). We also propose a random sampling approach to efficiently compute the preconditioned data. 


In summary, our contributions include: 
(i) we present a theory on  data preconditioning for the regularized loss optimization by introducing an appropriate  data preconditioner (Section~\ref{sec:theory}); 
(ii) we quantify the conditions under which the data preconditioning can reduce the condition number and therefore  boost the convergence of  the first-order optimization methods (c.f. equations~(\ref{cond:1}) and~(\ref{cond:2})); %
(iii) we present an efficient approach for computing the preconditioned data and validate the theory by experiments (Section~\ref{sec:imp},~\ref{sec:exp}). 

\section{Related Work}
We review some related work in this section. In particular, we survey some stochastic optimization algorithms that belong to the category of the first-order methods and discuss the dependence of their convergence rates on the condition number and the data. To facilitate our analysis, we decouple the dependence on the data from the  condition number. Henceforth, we denote by $R$ the upper bound of  the data norm, i.e.,  $\|\x\|_2\leq R$,  and by $L$ the Lipschitz constant of the scalar loss function $\ell(z, y)$ or its gradient $\ell'(z,y)$ with respect to $z$ depending the smoothness of the loss function. Then the gradient  w.r.t $\w$ of the loss function is bounded by  $\|\nabla_{\w} \ell(\w^{\top}\x, y)\|_2= \|\ell'(\w^{\top}\x, y)\x\|_2\leq LR$ if $\ell(z,y)$ is a $L$-Lipschitz continuous non-smooth function. Similarly, the second order gradient can be bounded by $\|\nabla_{\w}^2\ell(\w^{\top}\x, y)\|_2=\|\ell''(\w^{\top}\x, y)\x\x^{\top}\|_2\leq LR^2$ assuming  $\ell(z, y)$ is a $L$-smooth function. As a result the condition number for a $L$-Lipschitz continuous scalar loss function is $L^2R^2/\lambda$ and is $LR^2/\lambda$ for a $L$-smooth loss function. In the sequel, we will refer to $R$, i.e., the upper bound of the data norm as the data ingredient of the condition number, and refer to $L/\lambda$ or $L^2/\lambda$, i.e., the ratio of the Lipschitz constant to the strong convexity modulus as the functional ingredient of the condition number. The analysis in  Section~\ref{sec:theory} and~\ref{sec:imp} will exhibit how the data preconditioning affects the two ingredients. 

Stochastic gradient descent is probably the most popular algorithm in stochastic optimization. Although many variants of SGD have been developed, the simplest SGD for solving the problem~(\ref{eqn:p}) proceeds as:
\begin{align*}
\w_{t} = \w_{t-1}  - \eta_t \left(\nabla \ell(\w_{t-1}^{\top}\x_{i_t}, y_{i_t})+\lambda \w_{t-1}\right),
\end{align*}
where $i_t$ is randomly sampled from $\{1,\ldots, n\}$ and $\eta_t$ is an appropriate step size. The value of the step size $\eta_t$ depends on the strong convexity modulus of the objective function. If the loss function is a Lipschitz continuous function, the value of $\eta_t$ can be set to $1/(\lambda t)$~\cite{DBLP:conf/icml/Shamir013} that yields a convergence rate of $O\left(\frac{R^2L^2}{\lambda T}\right)$ with a proper averaging scheme. 
It has been shown that SGD achieves  the minimax optimal convergence rate for a non-smooth loss function~\cite{DBLP:conf/icml/Shamir013};  however, it only yields a sub-optimal convergence for a smooth loss function (i.e., $O(1/\sqrt{T})$) in terms of $T$. The curse of decreasing step size
 is the major reason that leads to the slow convergence. On the other hand, the decreasing step size is necessary due to the large variance of the stochastic gradient when approaching the optimal solution. 

Recently, there are several works dedicated to improving the convergence rate for a smooth loss function. The motivation is to reduce the variance of the stochastic gradient so as to use a constant step size like the full gradient method. We briefly mention several pieces of works. \cite{DBLP:conf/nips/RouxSB12} proposed a stochastic average gradient (SAG) method, which maintains an averaged stochastic gradient summing from gradients on all examples and updates a randomly selected component using the current solution. \cite{NIPS2013_4937,NIPS2013_4940} proposed accelerated SGDs using predicative variance reduction. The key idea is to use a mix of  stochastic gradients and  a full gradient. The two works share a similar idea that the algorithms compute a full gradient every certain iterations and construct an unbiased stochastic gradient using the full gradient and the gradients on one example. Stochastic dual coordinate ascent (SDCA)~\cite{DBLP:journals/jmlr/Shalev-Shwartz013} is another stochastic optimization algorithm that enjoys a fast convergence rate  for smooth loss functions. Unlike SGD types of algorithms, SDCA works on the dual variables and at each iteration it samples one instance and updates the corresponding dual variable by increasing the dual objective. It was shown in~\cite{NIPS2013_4937} that SDCA also achieves a variance reduction. Finally, all these algorithms have a comparable linear convergence for smooth loss functions with the iteration complexity being  characterized by $O\left(\left(n + \frac{R^2L}{\lambda}\right)\log(\frac{1}{\epsilon})\right)$~\footnote{The stochastic algorithm in \cite{NIPS2013_4940} has a quadratic dependence on the condition number.}. 

While most previous works target on improving the convergence rate for a better dependence on  the  number of iterations $T$, they have innocently ignored the fact of condition number. It has been observed when the condition number is very large, SGD suffers from a strikingly slow convergence due to that the step size $1/(\lambda t)$ is too large at the beginning of the iterations. The condition number is also an obstacle that prevents the scaling-up of the variance-reduced stochastic algorithms, especially when exploring the mini-batch technique. For instance,~\cite{NIPS2013_4938} proposed a mini-batch SDCA in which the iteration  complexity can be improved from  $O(\frac{n}{\sqrt{m}})$ to $O(\frac{n}{m})$ if the condition number is reduced from $n$ to $n/m$, where $m$ is the size of the mini-batch.  

Recently, there is a resurge of interest in importance sampling for stochastic optimization methods,  aiming to reduce  the condition number.  For example, Needell et al.~\cite{journals/corr/NeedellSW13} analyzed SGD with importance sampling  for strongly convex objective that is composed of individual smooth functions, where the sample for computing a stochastic gradient is drawn from a distribution with probabilities proportional to smoothness parameters of individual smooth functions. They showed that importance sampling can lead to a speed-up, improving the iteration complexity  from a quadratic dependence on the conditioning $(L/\lambda)^2$ (where $L$ is a bound on the smoothness and $\lambda$ on the strong convexity) to a linear dependence on $L/\lambda$. \cite{DBLP:journals/corr/ZhaoZ14,Xiao:arXiv1403.4699} analyzed the effect of importance sampling for stochastic mirror descent, stochastic dual coordinate ascent and stochastic variance reduced gradient method, and showed reduction on the condition number in the iteration complexity. However, all of these works could still suffer from very small strong convexity parameter $\lambda$ as in~(\ref{eqn:p}). Recently \cite{DBLP:conf/nips/BachM13}  provided a new analysis of the averaged stochastic gradient algorithm for a smooth objective function with a constant step size. They established  a convergence rate of $O(1/T)$ without suffering from the small strong convexity modulus. It has been observed  by empirical studies that  it could outperform SAG for solving least square regression and logistic regression. However, our experiments demonstrate that with data preconditioning the convergence of SAG can be substantially improved and better than that of  \cite{DBLP:conf/nips/BachM13}'s algorithm. More discussions can be found in the end of the subsection~\ref{sec:cond}. 

In recent years,  the idea of data preconditioning has been deployed in lasso~\cite{Jia:arXiv1208.5584,conf/recomb/HuangJ11,paul2008,NIPS2013_5104} via pre-multiplying the data matrix $X$ and the response vector $y$ by suitable matrices $P_X$ and $P_y$, to improve the support recovery properties. It was also brought to our attention that in ~\cite{DBLP:journals/abs/1502.03571} the authors applied data preconditioning to overdetermined $\ell_p$ regression problems and exploited SGD for the preconditioned problem. The big difference between our work and these works is that  we place emphasis  on applying data preconditioning to first-order stochastic optimization algorithms for solving the RLM problem in~(\ref{eqn:p}). Another remarkable difference  between the present work and these works is that in our study data preconditioning only applies to the feature vector $\x$ not the response vector $y$.  \

We also note that data preconditioning exploited in this work is different from preconditioning in some optimization algorithms that transforms the gradient by a preconditioner matrix or an adaptive matrix~\cite{6126441,Duchi:2011:ASM:1953048.2021068}. It is also different from the Newton method that multiplies the gradient by the inverse of the Hessian matrix~\cite{Boyd:2004:CO:993483}. As a comparison, the preconditioned data  can be computed offline and  the computational overhead can be made as minimal  as possible by using a large computer cluster with parallel computations. Unlike most previous works,  we strive to improve the convergence rate from the angle of reducing the condition number. We present a theory that characterizes  the conditions when the proposed data preconditioning can improve the convergence compared to the one without using data preconditioning. The contributed theory and technique act as an additional flavoring in the stochastic optimization that could improve the convergence speed. 

\section{Preliminaries}
In this section, we briefly introduce some key definitions  that are useful throughout the paper and then discuss a naive approach of applying data preconditioning for the RLM problem.

\begin{definition}
A function $f(\x): \R^d \rightarrow \R$ is a $L$-Lipschitz continuous function w.r.t a norm $\|\cdot\|$, if
\begin{align*}
|f(\x_1)- f(\x_2)|\leq L\|\x_1-\x_2\|,\forall \x_1,\x_2.
\end{align*}
\end{definition}
\begin{definition}\label{def:2}
A convex function $f(\x):\R^d\rightarrow\R$ is $\beta$-strongly convex w.r.t a norm $\|\cdot\|$, if for any $\alpha\in[0, 1]$
\begin{align*}
f(\alpha\x_1+ (1-\alpha)\x_2)\leq \alpha f(\x_1) + (1-\alpha) f(\x_2) - \frac{1}{2}\alpha(1-\alpha)\beta\|\x_1-\x_2\|^2, \forall \x_1,\x_2.
\end{align*}
where $\beta$ is also called the strong convexity modulus of $f$.
When $f(\x)$ is differentiable, the strong convexity is equivalent to
\begin{align*}
f(\x_1) \geq f(\x_2) + \langle \nabla f(\x_2), \x_1-\x_2\rangle + \frac{\beta}{2}\|\x_1-\x_2\|^2, \forall \x_1,\x_2.
\end{align*}
\end{definition}
\begin{definition}
A function $f(\x):\R^d\rightarrow\R$ is $L$-smooth w.r.t a norm $\|\cdot\|$, if it is differentiable and its gradient is $L$-Lipschitz continuous, i.e., 
\[
\|\nabla f(\x_1) - \nabla f(\x_2)\|_*\leq L\|\x_1 - \x_2\|, \forall \x_1, \x_2
\]
where $\|\cdot\|_*$ denotes the dual norm of $\|\cdot\|$, 
or equivalently
\begin{align*}
f(\x_1) \leq f(\x_2) + \langle \nabla f(\x_2), \x_1-\x_2\rangle + \frac{L}{2}\|\x_1-\x_2\|^2, \forall \x_1,\x_2.
\end{align*}
\end{definition}
In the sequel, we use the standard Euclidean norm to define Lipschitz and strongly convex functions. Examples of smooth loss functions include the logistic loss $\ell(\w; \x, y) = \log(1+\exp(-y\w^{\top}\x))$ and the square loss $\ell(\w; \x, y) = \frac{1}{2}(\w^{\top}\x-y)^2$. The $\ell_2$ norm regularizer $\frac{\lambda}{2}\|\w\|_2^2$ is a $\lambda$-strongly convex function. 


Although the proposed data preconditioning can be applied to boost any first-order methods, we will restrict our attention to the stochastic gradient methods, which share the following updates for~(\ref{eqn:p}) : 
\begin{align}\label{eqn:sgd}
\w_{t} = \w_{t-1}  - \eta_t \left(g_t(\w_{t-1})+\lambda \w_{t-1}\right),
\end{align}
where $g_t(\w_{t-1})$ denotes a stochastic gradient of the loss that depends on the original data representation. For example, the vanilla SGD for optimizing non-smooth loss uses $g_t(\w_{t-1})= \nabla\ell(\w^{\top}_{t-1}\x_{i_t}; y_{i_t})\x_{i_t}$, where $i_t$ is randomly sampled. SAG and SVRG  use a particularly designed stochastic gradient for minimizing a  smooth loss.  

A straightforward approach by exploring data preconditioning for the solving problem in~(\ref{eqn:p}) is by variable transformation. Let $P$ be a symmetric  non-singular matrix under consideration. Then we can cast the problem in~(\ref{eqn:p}) into:
\begin{align}\label{eqn:P1}
\min_{\u\in\R^d} \frac{1}{n}\sum_{i=1}^n\ell(\x_i^{\top}P^{-1}\u, y_i) + \frac{\lambda}{2}\|P^{-1}\u\|_2^2,
\end{align}
which could be implemented by preconditioning the data $\xh_i = P^{-1}\x_i$. Applying the stochastic gradient methods to the problem above we have the following update:
\begin{align*}
\u_{t} = \u_{t-1}  - \eta_t \left(g_t(\u_{t-1})+\lambda P^{-2}\u_{t-1}\right),
\end{align*}
where $g_t(\u_{t-1})$ denotes a stochastic gradient of the loss that depends on the transformed data representation.
However, there are two difficulties limiting the applications of the technique. First, what is an appropriate data preconditioner $P^{-1}$? Second, at each step we need to compute $P^{-2}\u_{t-1}$, which might add a significant cost ($O(d^2)$ if $P^{-2}$ is pre-computed and is a  dense matrix) to each iteration. To address these issues, we present a theory in the next section. In particular, we tackle three major questions: (i) what is the appropriate data preconditioner for the first-order  methods to minimize the regularized loss as in~(\ref{eqn:p}); (ii) under what conditions (w.r.t the data and the loss function)  the data preconditioning can boost the convergence; and (iii) how to efficiently compute the preconditioned data.

\section{Theory}\label{sec:theory}
\subsection{Data preconditioning for Regularized Loss Minimization}
The first question that we are about to address is ``what is the  condition on the loss function 
in order for data preconditioning to take effect''. The question turns out to be related to how we construct the preconditioner. We  are inclined to give the condition first and explain it when we construct the preconditioner. To facilitate our discussion, we assume that the first argument of the loss function is bounded by $r$, i.e., $|z|\leq r$. We defer the discussion on the value of $r$ to the end of this section. 
The condition for the loss function given below is complimentary to the property of Lipschitz continuity. 
\begin{assumption}\label{ass:1}
The scalar loss function $\ell(z, y)$  w.r.t $z$ satisfies $\ell''(z,y)\geq \beta$ for $|z|\leq r$ and  $\beta>0$. 
\end{assumption}
Below we discuss several important loss functions used in machine learning and statistics that have such a property. 
\begin{itemize}
\item Square loss. The square loss $\ell(z,y)=\frac{1}{2}|y-z|^2$ has been  used in ridge regression and classification. It is clear that the square loss satisfies the assumption for any $z$ and $\beta=1$. 
\item Logistic loss. The logistic loss $\ell(z,y) = \log(1+\exp(-zy))$ where $y\in\{1,-1\}$ is used in logistic regression for classification. We can compute the second order gradient by $\ell''(z,y) = \sigma(yz)(1-\sigma(yz))$, where $\sigma(z) = 1/(1+\exp(-z))$ is the sigmoid function. Then it is not difficult to show that when $|z|\leq r$, we have $\ell''(z,y)\geq \sigma(r)(1-\sigma(r))$. Therefore the assumption~(\ref{ass:1}) holds for $\beta(r) = \sigma(r)(1-\sigma(r).$
\item Possion regression loss. In statistics,  Poisson regression is a form of regression analysis used to model count data and contingency tables. The equivalent loss function is given by $\ell(z,y) = \exp(z) - yz$. Then $\ell''(z,y) = \exp(z)\geq \exp(-r)$ for $|z|\leq r$. Therefore the assumption~(\ref{ass:1}) hold for $\beta(r) = \exp(-r)$.
\end{itemize}
It is notable that the Assumption~\ref{ass:1} does not necessarily indicate that the entire loss $(1/n)\sum_{i=1}^n\ell(\w^{\top}\x_i, y_i)$ is a strongly convex function w.r.t $\w$ since the second order gradient, i.e.,  $\frac{1}{n}\sum_{i=1}^m\ell''(\w^{\top}\x_i, y_i)\x_i\x_i^{\top}$ is not necessarily lower bounded by a positive constant. Therefore the introduced condition does not change the convergence rates that we have discussed. 
The construction of the data preconditioner  is motivated by the following observation. Given $\ell''(z,y)\geq \beta$ for any $|z| \leq r$, we can define a  new loss function $\phi(z, y)$ by
\begin{align*}
    \phi(z, y) = \ell(z, y) - \frac{\beta}{2}z^2,
\end{align*}
 and we can easily show that $\phi(z,y)$ is convex for $|z|\leq r$. Using $\phi(z, y)$, we can transform the problem in~(\ref{eqn:p}) into:
\begin{align*}
\min_{\w\in\R^d} \frac{1}{n}\sum_{i=1}^n\phi(\w^{\top}\x_i, y_i)& + \frac{\beta}{2}\w^{\top}\frac{1}{n}\sum_{i=1}^n\x_i\x_i^{\top}\w+ \frac{\lambda}{2}\|\w\|_2^2.
\end{align*}
Let  $C=\frac{1}{n}\sum_{i=1}^n\x_i\x_i^{\top}$ denote  the sample covariance matrix.  We define a smoothed covariance matrix $H$ as
\[
H = \rho I + \frac{1}{n}\sum_{i=1}^n \x_i\x_i^{\top} = \rho I + C,
\]
where  $\rho = \lambda/\beta$. Thus, the transformed problem becomes 
 \vspace*{-0.1in}
\begin{align}\label{eqn:np}
\min_{\w\in\R^d} \frac{1}{n}\sum_{i=1}^n\phi(\w^{\top}\x_i, y_i)& + \frac{\beta}{2}\w^{\top}H\w.
\end{align}
Using the variable transformation $\v \leftarrow  H^{1/2}\w$, the above problem is equivalent to  
\begin{align}\label{eqn:np}
\min_{\v\in\R^d} \frac{1}{n}\sum_{i=1}^n\phi(\v^{\top}H^{-1/2}\x_i, y_i)& + \frac{\beta}{2}\|\v\|_2^2.
\end{align}
It can be shown that the optimal  value of the above preconditioned  problem is equal to that  of the original problem~(\ref{eqn:p}).  As a matter of fact, so far we  have constructed a data preconditioner as given by $P^{-1}= H^{-1/2}$ that transforms the original feature vector $\x$ into a new vector $H^{-1/2}\x$. It is worth noting that the data preconditioning  $H^{-1/2}\x$ is similar  to the ZCA whitening transformation, which transforms the data using the covariance matrix, i.e.,  $C^{-1/2}\x$ such that the data has identity covariance matrix. Whitening transformation has found many applications in image processing~\cite{books/daglib/0019304}, 
 and it is also  employed in independent component analysis~\cite{Hyvarinen:2000:ICA:351654.351659} and optimizing deep neural networks~\cite{journals/jmlr/RanzatoKH10,lecun-efficient-backprop-1998}. A similar idea has been used  decorrelation of the covariate/features in statistics~\cite{mardia1979multivariate}. Finally, it is notable that when original data is sparse the preconditioned data may become dense, which may increase  the per-iteration cost. It would  pose stronger conditions for the data preconditioning to take effect. In our experiments, we focus on dense data sets.

\subsection{Condition Number}\label{sec:cond}
Besides the data, there are two additional alterations: (i) the strong convexity modulus is changed  from $\lambda$ to $\beta$ and (ii) the loss function becomes  $\phi(z, y)=\ell(z,y) - \frac{\beta}{2}z^2$. Before discussing  the convergence rates of the first-order optimization methods for solving the preconditioned problem in~(\ref{eqn:np}), we elaborate on  how the two ingredients of the condition number are affected:  (i) the functional ingredient namely the ratio of the  Lipschitz constant of the loss function to the strong convexity modulus  and (ii) the data ingredient namely the upper bound of the data norm.  We first analyze the change of the functional ingredient as summarized  in the following lemma. 
\begin{lemma}\label{lem:1}
If $\ell(z,y)$ is a $L$-Lipschitz continuous function, then $\phi(z, y)$ is $(L+\beta r)$-Lipschitz continuous for $|z|\leq r$. If $\ell(z,y)$ is a $L$-smooth function, then $\phi(z, y)$ is a $(L- \beta)$-smooth function. 
\end{lemma}
\begin{proof}
If $\ell(z,y)$ is a $L$-Lipschitz continuous function, the new function $\phi(z, y)$ is a $(L  + \beta r)$-Lipschitz continuous for $|z|\leq r$ because
\begin{align*}
|\phi(z_1, y) - \phi(z_2, y)|&\leq L|z_1-z_2|  + \frac{\beta}{2}|z_1 -z_2|^2\\
&\leq (L + \beta r)|z_1 - z_2|
\end{align*}
If $\ell(z, y)$ is a $L$-smooth function, then the following equality holds~\cite{opac-b1104789}
\[
\langle\ell'(z_1,y) - \ell'(z_2,y), z_1-z_2\rangle\leq L|z_1-z_2|^2. 
\]
By the definition of $\phi(z,y)$, we have
\begin{align*}
&\langle\phi'(z_1, y) + \beta z_1 - \phi'(z_2,y) - \beta z_2, z_1-z_2\rangle\leq L|z_1-z_2|^2
\end{align*}
Therefore 
\begin{align*}
&\langle\phi'(z_1, y) - \phi'(z_2,y),  z_1-z_2\rangle \leq (L-\beta)|z_1-z_2|^2 
\end{align*}
which implies $\phi(z,y)$ is a $(L-\beta)$-smooth function~\cite{opac-b1104789}. 
\end{proof}

Lemma~\ref{lem:1} indicates that after the data preconditioning the functional ingredient becomes $\displaystyle (L+\beta r)^2/\beta$ for a $L$-Lipschitz continuous non-smooth loss function and $\displaystyle{(L-\beta)}/{\beta}$ for a $L$-smooth function. Next, we analyze  the upper bound of the preconditioned data $\widehat\x = H^{-1/2}\x$. Noting that $\|\widehat\x\|_2^2=\x^{\top}H^{-1}\x$, in what follows we will focus on bounding $\max_i\x_i^{\top}H^{-1}\x_i$.  We first derive and discuss  the bound of  the expectation $\E_{i}[\x_{i}^{\top}H^{-1}\x_{i}]$ treating $i$ as a random variable in $\{1,\ldots, n\}$, which is useful in proving the convergence bound of the objective in expectation. Many discussions also carry  over to the upper bound for individual data. 
Let $ \frac{1}{\sqrt{n}}X=\frac{1}{\sqrt{n}}(\x_1,\cdots, \x_n)=U\Sigma V^{\top}$ be the singular value decomposition of $X$, where $U\in\R^{d\times d}, V\in\R^{n\times d}$ and $\Sigma=diag(\sigma_1,\ldots, \sigma_d), \sigma_1\geq \ldots\geq \sigma_d$, then $C = U\Sigma^2U^{\top}$ is the eigen-decomposition of $C$. Thus, we have
\begin{align}
\E_{i}[\x^{\top}_{i}H^{-1}\x_i] =\frac{1}{n}\sum_{i=1}^n\x_i^{\top}H^{-1}\x_i= tr(H^{-1}C)=\sum_{i=1}^d\frac{\sigma^2_i}{\sigma^2_i +\rho}  \buildrel \Delta \over = \gamma(C, \rho).
\end{align}
where the expectation is taken over the randomness in the index $i$, which is also the source of randomness in stochastic gradient descent methods. 
We refer to $\gamma(C,\rho)$ as the numerical rank of $C$ with respect to $\rho$.   
The first observation is that $\gamma(C,\rho)$ is a monotonically  decreasing function in terms of $\rho$. 
It is straightforward to show that if $X$ is low rank, e.g., $rank(X)=k\ll d$, then $\gamma(C,\rho)<k$. 
If $C$ is full rank, the value of $\gamma(C,\rho)$ will be affected by the decay of its eigenvalues. Bach~\cite{DBLP:conf/colt/Bach13} has derived  the order of $\gamma(C,\rho)$ in $\rho$ under two different decays of the eigenvalues of $C$. The following proposition summarizes the order of $\gamma(C,\rho)$ under two different decays of the eigenvalues. 
\begin{prop}\label{prop:1}
If  the eigenvalues of $C$ follow a polynomial decay $\sigma_i^2 = i^{-2\tau}, \tau\geq1/2$, then $\gamma(C,\rho)\leq O(\rho^{-1/(2\tau)})$, and if the eigenvalues of $C$ satisfy an exponential decay $\sigma_i^2=e^{-\tau i}$, then $\gamma(C, \rho)\leq O\left(\log\left(\frac{1}{\rho}\right)\right)$. 
\end{prop}
For completeness, we include the proof in the Appendix A.  In statistics~\cite{hastie01statisticallearning},  $\gamma(C,\rho)$ is also referred to as the effective degree of freedom. In order to prove high probability bounds, we have to derive the upper bound for individual $\x_i^{\top}H^{-1}\x_i$. 
To this end, we introduce the following measure to quantify the incoherence of $V$. 
\begin{definition} The generalized incoherence measure of an orthogonal matrix $V\in\R^{n\times d}$ w.r.t to $(\sigma_1^2,\ldots, \sigma^2_d)$ and $\rho>0$ is 
\begin{equation}
\mu(\rho)= \max\limits_{1 \leq i \leq n} \frac{n}{\gamma(C, \rho)}\sum_{j=1}^d\frac{\sigma^2_j}{\sigma^2_j +\rho}V_{ij}^2.
\end{equation}
\end{definition}
Similar to the incoherence measure introduced in the compressive sensing theory~\cite{citeulike:1227656}, the generalized incoherence also measures the degree to which the rows in $V$are correlated with the canonical bases. We can also establish the relationship between the two incoherence measures. The incoherence of an orthogonal matrix $V\in\R^{n\times n}$ is defined  as $\mu=\max_{ij}\sqrt{n}V_{ij}$~\cite{citeulike:1227656}. With simple algebra, we can show that $\mu(\rho)\leq \mu^2$. Since $\mu\in[1,\sqrt{n}]$, therefore $\mu(\rho)\in[1, n]$. 
Given the definition of $\mu(\rho)$, we have the following lemma on the upper bound of $\x_i^{\top}H^{-1}\x_i$. 
\begin{lemma}\label{lem:2} 
$\x^{\top}_i H^{-1}\x_i\leq \mu(\rho)\gamma(C,\rho), \quad i=1,\ldots, n$.
\end{lemma}
\begin{proof}
Noting the SVD of $X=\sqrt{n}U\Sigma V^{\top}$, we have $\x_i = \sqrt{n}U\Sigma V^{\top}_{i,*}$, where $V_{i,*}$ is the $i$-th row of $V$, we have
\begin{align*}
\x_i^{\top}H^{-1}\x_i &=n  V_{i,*}\Sigma U^{\top}U(\Sigma + \rho I)^{-1}U^{\top}U\Sigma V_{i,*}^{\top}\\
& = n V_{i,*}\Sigma(\Sigma + \rho I)^{-1}\Sigma V_{i,*}^{\top} = n \sum_{j=1}^d\frac{\sigma_j^2}{\sigma_j^2 + \rho} V^2_{ij}
\end{align*}
Following the definition of $\mu(\rho)$, we can complete the proof 
\[
\max_{1\leq i\leq n}\x_i^{\top}H^{-1}\x_i\leq \mu(\rho)\gamma(C,\rho)
\]
\end{proof}

The theorem below states  the condition number of the preconditioned problem~(\ref{eqn:np}).
\begin{thm}\label{thm:1}
If $\ell(z,y)$ is a $L$-Lipschitz continuous  function satisfying the condition in Assumption~\ref{ass:1}, then the condition number of the optimization problem in~(\ref{eqn:np}) is bounded by  $\frac{(L+\beta r)^2\mu(\rho)\gamma(C,\rho)}{\beta}$, where $\rho=\lambda/\beta$. If $\ell(z, y)$ is a $L$-smooth function satisfying the condition in Assumption~\ref{ass:1}, then the condition number of~(\ref{eqn:np}) is $\frac{(L-\beta)\mu(\rho)\gamma(C,\rho)}{\beta}$. 
\end{thm}
Following the above theorem and previous discussions on the condition number, we have the following observations about when {\bf the data preconditioning can reduce the condition number}. 
\begin{observation}\label{cond:2}
\begin{enumerate}
\item 
If $\ell(z, y)$ is a $L$-Lipschitz continuous function and 
\begin{equation}\label{cond:1}
\frac{\lambda (L+\beta r)^2}{\beta L^2}\leq \frac{R^2}{\mu(\rho)\gamma(C,\rho)}   
\end{equation}
where $r$ is the upper bound of predictions $z=\w_t^{\top}\x_i$ during optimization, then the proposed data preconditioning can reduce the condition number. 
\item  If $\ell(z, y)$ is $L$-smooth and 
\begin{equation}\label{cond:2}
 \frac{\lambda}{\beta} - \frac{\lambda}{L}\leq \frac{R^2}{\mu(\rho)\gamma(C,\rho)}
\end{equation}
 then the proposed data preconditioning can reduce the condition number. 
\end{enumerate}
\end{observation}
\paragraph{Remark 1: } In the above conditions~((\ref{cond:1}) and (\ref{cond:2})), we make explicit  the effect from the loss function and the data. In the right hand side,  the quantity $R^2/\mu(\rho)\gamma(C,\rho)$ measures the ratio between the maximum norm of the original data and that of the preconditioned data.  The left hand side depends on the property of the loss function and the value of $\lambda$. Due to the unknown value of $r$ for non-smooth optimization, we first discuss the indications of the condition for the smooth loss function and comment on the value of $r$ in Remark 2.  Let us consider $\beta, L\approx \Theta(1)$ (e.g. in ridge regression or regularized least square classification) and $\lambda =\Theta(1/n)$. Therefore $\rho=\lambda/\beta=\Theta(1/n)$. The condition in~(\ref{cond:2}) for the smooth loss requires the ratio between the maximum norm of the original data and that of the preconditioned data is larger than $\Theta(1/n)$. If the eigenvalues of the covariance matrix follow an {\it exponential decay}, then $\gamma(C,\rho)=\Theta(1)$ and the condition indicates that 
\[
\mu(\rho)\leq \Theta(n R^2),
\] which can be satisfied easily if $R>1$ due to the fact $\mu(\rho)\leq n$. If the eigenvalues follow a {\it polynomial decay} $i^{-2\tau}, \tau\geq 1/2$, then $\gamma(C,\rho)\leq O(\rho^{-1/(2\tau)}) = O(n^{1/(2\tau)})$, then the condition indicates that 
\[
\mu(\rho)\leq O(n^{1-\frac{1}{2\tau}}R^2),
\] which means the faster the decay of the eigenvalues, the easier for the condition to be satisfied.  Actually, several previous works~\cite{DBLP:conf/uai/TalwalkarR10,DBLP:journals/corr/abs-1303-1849,DBLP:conf/nips/YangJ14} have studied the coherence measure  and demonstrated that it is not rare to have  a small coherence measure for real data sets, making the above inequality easily satisfied. 

If $\beta$ is a small value (e.g., in logistic regression), then the satisfaction of the condition depends on the balance between the factors $\lambda, L, \beta, \gamma(C,\rho), \mu(\rho), R^2$. In practice, if $\beta, L$ is known we can always check the condition by calculating  the ratio between the maximum norm of the original data and that of the preconditioned data and comparing it with $\lambda/\beta - \lambda/L$. If $\beta$ is unknown, we can  take a trial and error method by tuning  $\beta$ to achieve the best performance. 

\paragraph{Remark 2:} Next, we comment on the value of $r$ for non-smooth optimization. It was shown in~\cite{DBLP:journals/mp/Shalev-ShwartzSSC11} the optimal solution $\w_*$ to~(\ref{eqn:p}) can be bounded by $\|\w_*\|\leq O(\frac{1}{\sqrt{\lambda}})$. Theoretically we can ensure $|z|=|\w^{\top}\x|\leq {R}/{\sqrt{\lambda}}$ and thus $r^2\leq R^2/\lambda$. In the worse case $r^2={R^2}/{\lambda}$, the condition number of the preconditioned problem for non-smooth optimization is bounded by $O\left(\left(\frac{L^2}{\beta}+\frac{R^2}{\lambda\beta}\right)\mu(\rho)\gamma(C,\rho)\right)$. Compared to the original condition number $L^2R^2/\lambda$,  there may be no improvement for convergence. In practice, $\|\w_*\|_2$ could be much less than $1/\sqrt{\lambda}$ and therefore $r< R/\sqrt{\lambda}$, especially when $\lambda$ is very small. On the other hand, when $\lambda$ is too small the step sizes $1/(\lambda t)$ of SGD on the original problem at the beginning of iterations are extremely large, making the optimization unstable. This issue can be mitigated or eliminated by data preconditioning. 

\paragraph{Remark 3:} We can also analyze the straightforward approach by solving the preconditioned problem in~(\ref{eqn:P1})  using $P^{-1}=H^{-1/2}$. Then the problem becomes:
\begin{align}\label{eqn:P1H}
\min_{\u\in\R^d} \frac{1}{n}\sum_{i=1}^n\ell(\u^{\top}H^{-1/2}\x_i, y_i) + \frac{\lambda}{2}\u^{\top}H^{-1}\u,
\end{align}

The bound of the data ingredient follows the same analysis. The functional ingredient is $\tilde O\left(\frac{L (\sigma_1^2 + \rho)}{\lambda}\right)$ due to that $\lambda \u^{\top}H^{-1}\u\geq \lambda/(\sigma_1^2 + \rho)\|\u\|_2^2$. If $\lambda\ll\sigma_1^2$, then the condition number of the preconditioned problem still heavily depends on $1/\lambda$. Therefore, solving the naive preconditioned problem~(\ref{eqn:P1}) with $P^{-1}=H^{-1/2}$ may not boost the convergence, which is also verified in Section~\ref{sec:exp}  by experiments. 

\paragraph{Remark 4:} Finally, we use the example of SAG for solving least square regression to demonstrate the benefit of data preconditioning. Similar analysis carries on to other variance reduced stochastic optimization algorithms~\cite{NIPS2013_4937,DBLP:journals/jmlr/Shalev-Shwartz013}. When $\lambda=1/n$ the iteration complexity  of SAG would be dominated by $O(R^2n\log(1/\epsilon))$~\cite{DBLP:journals/corr/SchmidtRB13} -- tens of  epochs 
depending on the value of $R^2$. However, after data preconditioning the iteration complexity becomes $O(n\log(1/\epsilon))$ if $n\geq\hat R^2$, where $\hat R$ is the upper bound of the preconditioned data, which would be just few epochs. In comparison, Bach and Moulines' algorithm~\cite{DBLP:conf/nips/BachM13} suffers from an $O(\frac{d+R^2}{\epsilon})$ iteration complexity that could be much larger than $O(n\log(1/\epsilon))$, especially when required $\epsilon$ is small and $R$ is large.  Our empirical studies in Section~\ref{sec:exp} indeed verify these results. 



\subsection{Efficient Data Preconditioning}\label{sec:imp}
Now we proceed to address the third question, i.e., how to efficiently compute the preconditioned data. 
The data preconditioning using $H^{-1/2}$ needs to compute the square root inverse of $H$ times $\x$, which usually costs a time complexity of $O(d^3)$.  On the other hand, the computation of the preconditioned data for least square regression is as expensive as computing the closed form solution, which makes data preconditioning not attractive, especially for high-dimensional data. 
In this section, we analyze an efficient data preconditioning by random sampling. As a compromise, we might lose some gain in convergence.  
The key idea is to construct the preconditioner by sampling a subset of $m$ training data, denoted by $\widehat\D=\{\xh_1, \ldots, \xh_m\}$. Then we construct new loss functions for individual data as, 
\begin{align*}
&\psi(\w^{\top}\x_i, y_i) =\left\{\begin{array}{l}\ell(\w^{\top}\x_i, y_i) - \frac{\beta}{2}(\w^{\top}\x_i)^2,\:\text{ if } \x_i\in\widehat\D\\
\\
\ell(\w^{\top}\x_i, y_i),\:\text{ otherwise }\end{array}\right.
\end{align*}
We define $\hat\beta$ and $\hat\rho$ as 
\begin{equation}\label{eqn:rho}
\hat\beta = \frac{m}{n}\beta, \quad\quad \hat\rho = \frac{n}{m}\rho = \frac{n\lambda}{m\beta} = \frac{\lambda}{\hat\beta}
\end{equation}
Then we can show that the original problem is equivalent to 
 \begin{align}\label{eqn:np-2}
\min_{\v\in\R^d} \frac{1}{n}\sum_{i=1}^n\psi(\v^{\top}\Hh^{-1/2}\x_i, y_i)& + \frac{\hat\beta}{2}\|\v\|_2^2.
\end{align} 
where $\Hh = \hat\rho I + \frac{1}{m}\sum_{i=1}^m \xh_i \xh_i^{\top}$. Thus, $\Hh^{-1/2}\x_i$ defines the new preconditioned data. 
Below we show  how to efficiently  compute $\Hh^{-1}\x$. Let $\frac{1}{\sqrt{m}}\hat X = \hat U \hat\Sigma \hat V^{\top}$ be the SVD of $\hat X=(\xh_1,\ldots, \xh_m)$, where $\hat U\in\R^{d\times m}, \hat\Sigma=diag(\hat\sigma_1,\ldots,\hat\sigma_m)$. Then with simple algebra $\Hh^{-1/2}$ can be written as 
\begin{align*}
\Hh^{-1/2}&=(\hat\rho I  + \hat U\hat\Sigma^2\hat U^{\top})^{-1/2}= \hat\rho^{-1/2}I - \hat U \hat S\hat U^{\top},
\end{align*}
where $\hat S=diag(\hat s_1,\ldots, \hat s_m)$ and $\hat s_i =\hat \rho^{-1/2} - (\hat\sigma^2_i + \hat\rho)^{-1/2}$. 
Then the preconditioned data $\Hh^{-1/2}\x_i$ can be calculated by  $
\Hh^{-1/2}\x_i = \hat\rho^{-1/2}\x_i - \hat U(\hat S(\hat U^{\top}\x_i))$,  
which costs $O(md)$ time complexity. Additionally, the time complexity  for computing the SVD of $\hat X$ is $O(m^2d)$. Compared with the preconditioning with full data, the above procedure of preconditioning is much more efficient.  Moreover,  the calculation  of the preconditioned data given the SVD of $\hat X$ can be carried out on multiple machines to make the computational overhead as minimal as possible.

It is worth noting that the random sampling approach has been used previously to construct the stochastic Hessian~\cite{conf/icml/Martens10,journals/siamjo/ByrdCNN11}. Here, we analyze its impact on  the condition number. 
The same analysis about the Lipschitz constant of the loss function carries over to $\psi(z,y)$, except that $\psi(z,y)$ is  at most $L$-smooth if $\ell(z,y)$ is  $L$-smooth. The following theorem allows us to bound the norm of the preconditioned data using $\Hh$. 
\begin{thm}\label{thm:key}
Let $\hat\rho$ be defined in~(\ref{eqn:rho}). For any $\delta\leq 1/2$, If 
\[
m \geq \frac{2}{\delta^2}(\mu(\hat\rho)\gamma(C, \hat\rho) + 1)(t + \log d),
\] then with a probability $1 - e^{-t}$, we have
\begin{align*}
\x_i^{\top}\Hh^{-1}\x_i \leq(1+2\delta)\mu(\hat\rho)\gamma(C,\hat\rho), \forall i=1,\ldots, n
\end{align*}
\end{thm}

The proof of the theorem is presented in Appendix B. The theorem indicates that the upper bound of the preconditioned data is only scaled up by a small constant factor with an overwhelming probability compared to that using  all data points to construct the preconditioner under moderate conditions when the data matrix $X$ has a low coherence. 
 Before ending this section, we present a similar theorem to Theorem \ref{thm:1} for using the efficient data preconditioning. 
\begin{thm}\label{thm:3}
If $\ell(z,y)$ is a $L$-Lipschitz continuous  function satisfying the condition in Assumption~\ref{ass:1}, then the condition number of the optimization problem in~(\ref{eqn:np-2}) is bounded by  $\frac{(L+\beta r)^2\mu(\hat\rho)\gamma(C,\hat\rho)}{\hat\beta}$. If $\ell(z, y)$ is a $L$-smooth function satisfying the condition in Assumption~\ref{ass:1}, then the condition number of~(\ref{eqn:np-2}) is $\frac{L\mu(\hat\rho)\gamma(C,\hat\rho)}{\hat\beta}$. 
\end{thm} 
Thus, similar conditions can be established for the data preconditioning using $\Hh^{-1/2}$ to improve the convergence rate. Moreover, varying $m$ may exhibit a tradeoff between the two ingredients  understood as follows. Suppose the incoherence measure $\mu(\rho)$ is bounded by a constant. Since $\gamma(C,\hat\rho)$ is a monotonically decreasing function w.r.t $\hat\rho$, therefore $\gamma(C,\hat\rho)$ and the data ingredient $\x_i^{\top}\Hh^{-1}\x_i$ may increase as $m$ increases. On the other hand, the functional ingredient $L/\hat\beta$ would decrease as $m$ increases. 

\section{Experiments}\label{sec:exp}
\subsection{Synthetic Data}  We first present some simulation results to verify our theory. To control the inherent data properties (i.e, numerical rank and incoherence), we generate synthetic data. We first generate a standard Gaussian matrix $M\in\R^{d\times n}$ and then compute its SVD $M=USV^{\top}$. We use $U$ and $V$ as the left and right singular vectors to construct the data matrix $X\in\R^{d\times n}$. In this way, the incoherence measure of $V$ is a small constant (around $5$). We generate eigenvalues of $C$ following a polynomial decay $\sigma^2_i=i^{-2\tau}$ (poly-$\tau$) and an exponential decay $\sigma^2_i =\exp(-\tau i)$. Then we construct the data matrix $X=\sqrt{n}U\Sigma V ^{\top}$, where $\Sigma=diag(\sigma_1,\cdots, \sigma_d)$.

We first plot the condition number for the problem in ~(\ref{eqn:p}) and its data preconditioned problem  in~(\ref{eqn:np}) using $H^{-1/2}$ by assuming the Lipschitz constant $L=1$, varying the decay of the eigenvalues of the sample covariance matrix, and varying  the values of $\beta$ and  $\lambda$. To this end, we generate a synthetic data with $n=10^{5}, d=100$. The curves in Figure~\ref{fig:cond1} show the condition number vs the values of $\beta$ by varying the decay of the eigenvalues. It indicates that the data preconditioning can reduce the condition number for a broad range of values of $\beta$, the strong convexity modulus of the scalar loss function. The curves in Figure~\ref{fig:cond2} show a similar pattern of  the condition number vs the values of $\lambda$ by varying the decay of the eigenvalues. It also exhibits that the smaller the $\lambda$ the larger reduction in the condition number. 

\begin{figure}[t]\label{fig:cond}
\centering
\subfigure[fix $\lambda=10^{-5}, L=1$]{\label{fig:cond1}\includegraphics[scale=0.25]{{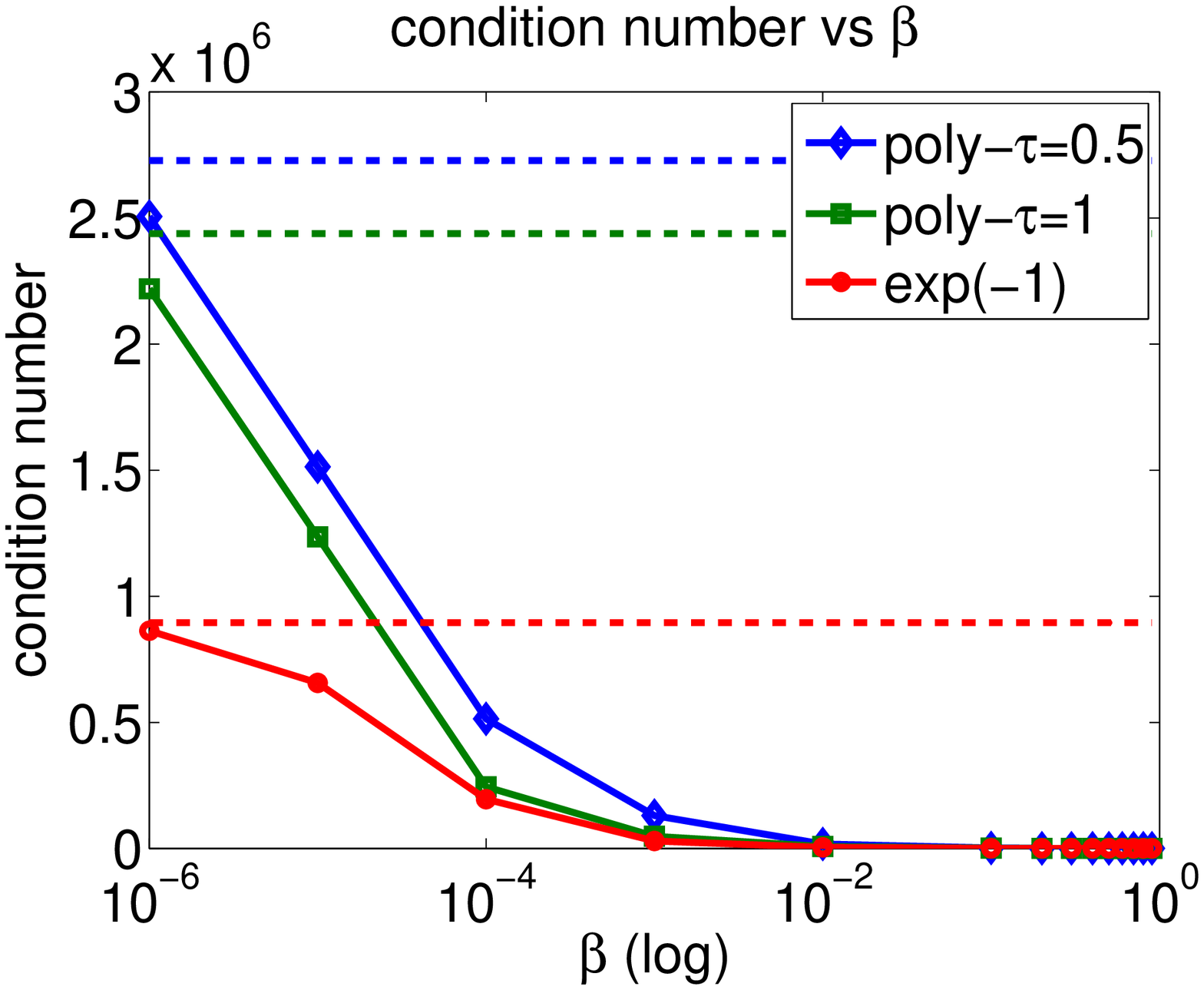}}}
\subfigure[fix $\beta=10^{-3}, L=1$]{\label{fig:cond2}\includegraphics[scale=0.25]{{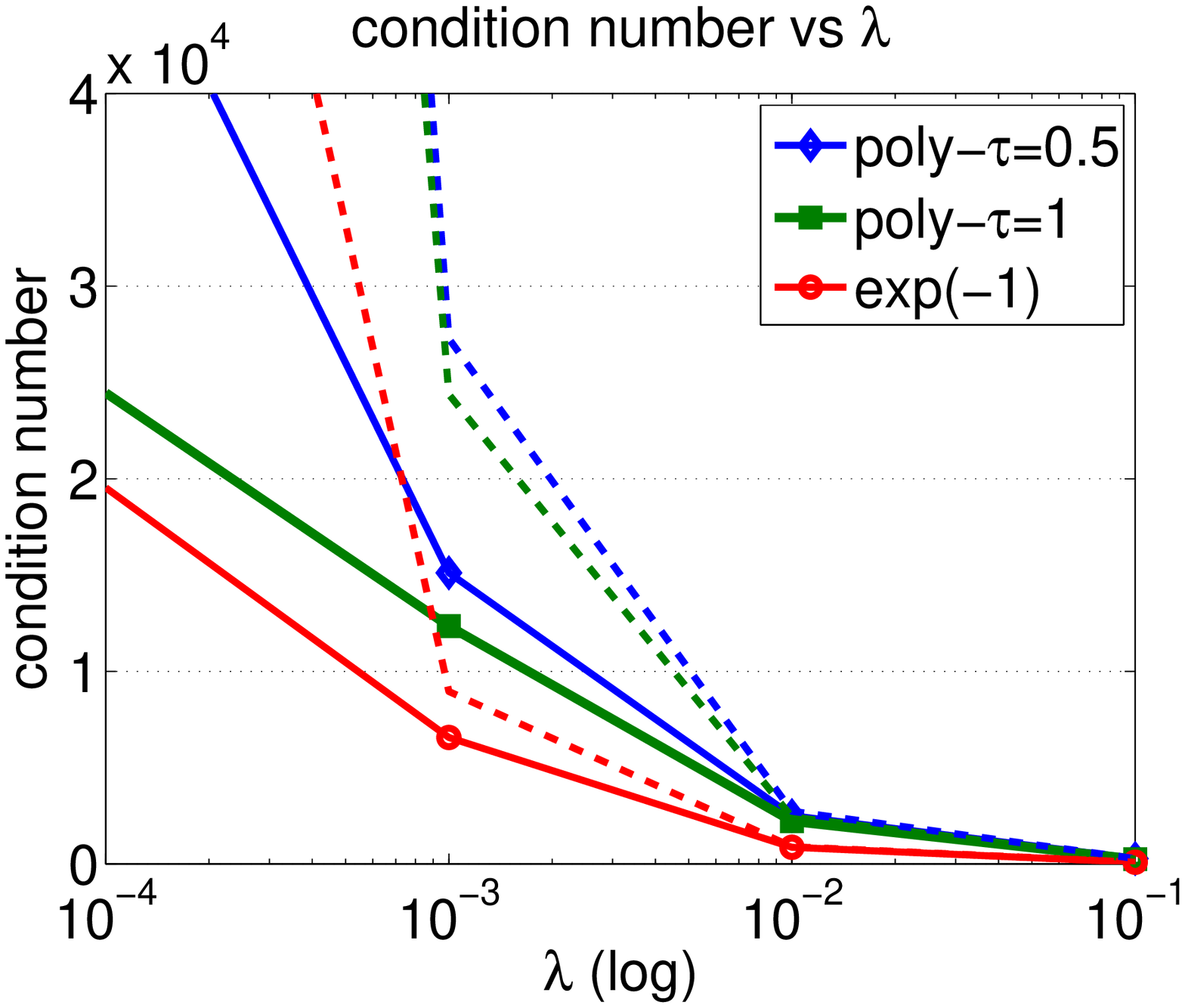}}}

\caption{Synthetic data: (a) compares the condition number of the preconditioned problem (solid lines) with that of  the original problem (dashed lines of the same color) by varying the value of $\beta$ (a property of the loss function) and varying the decay of the eigenvalues of the sample covariance matrix (a property of the data); (b) compares the condition number by varying the value of $\lambda$ (measuring the difficulty  of the problem) and varying the decay of the eigenvalues. 
} 
\end{figure}

\begin{figure}[t]
\centering
\subfigure[poly-$\tau$ ($0.5$), $\beta=0.99$]{\label{fig:2a}\includegraphics[scale=0.25]{{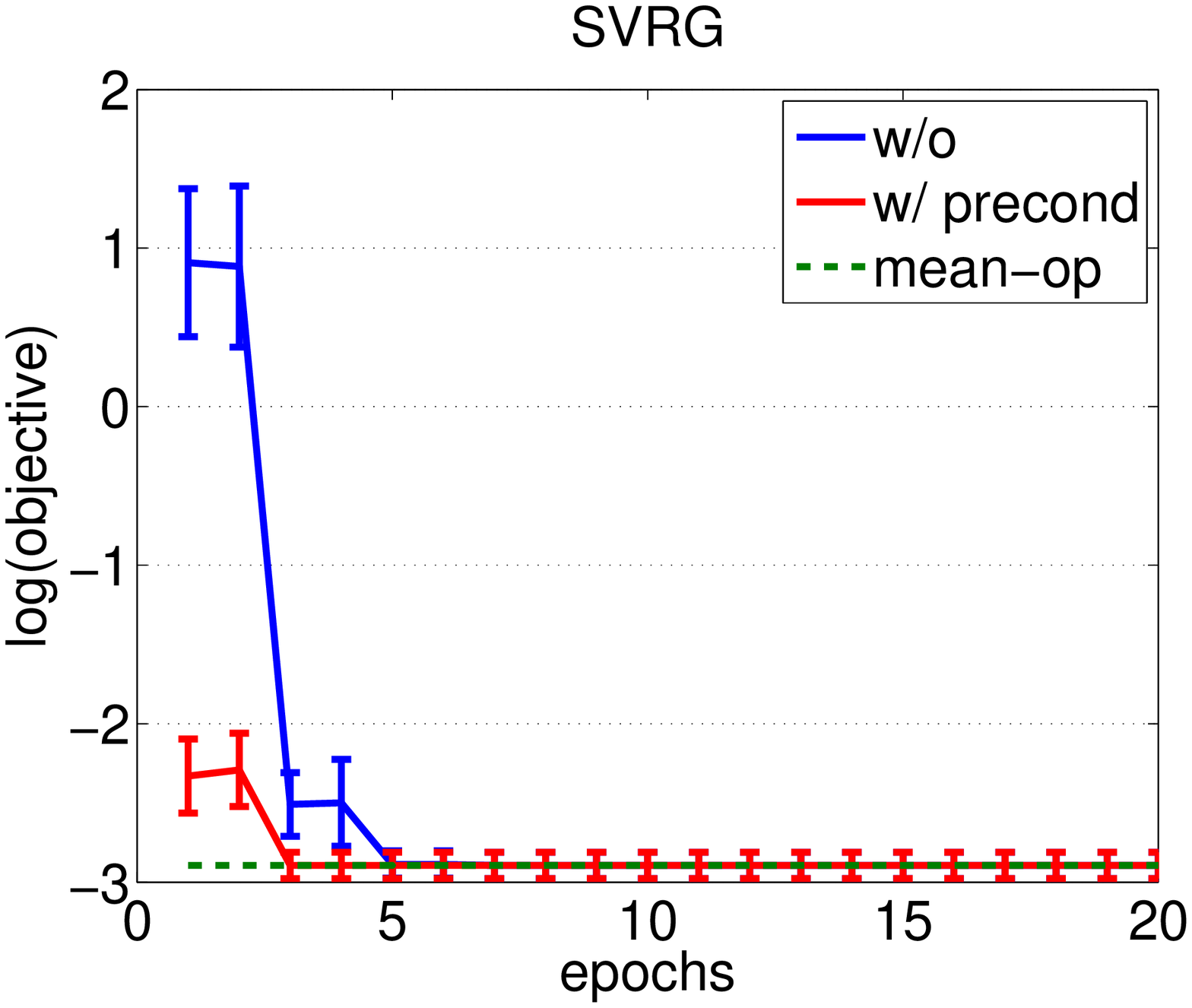}}}
\hspace*{0.15in}
\subfigure[poly-$\tau$ ($0.5$), $\beta=0.99$]{\includegraphics[scale=0.25]{{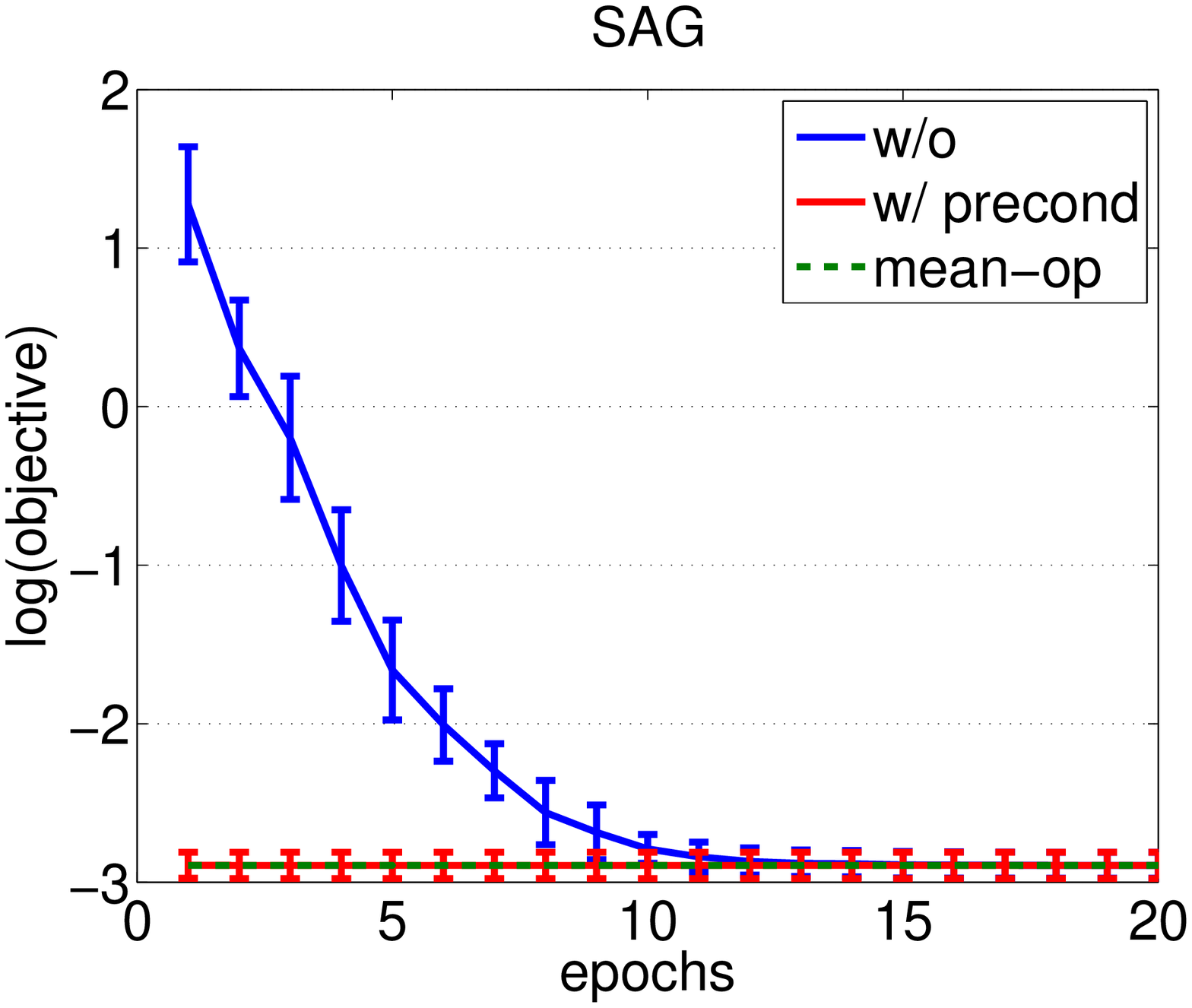}}}

\subfigure[poly-$\tau$ ($0.5$), $\lambda=10^{-5}$]{\includegraphics[scale=0.25]{{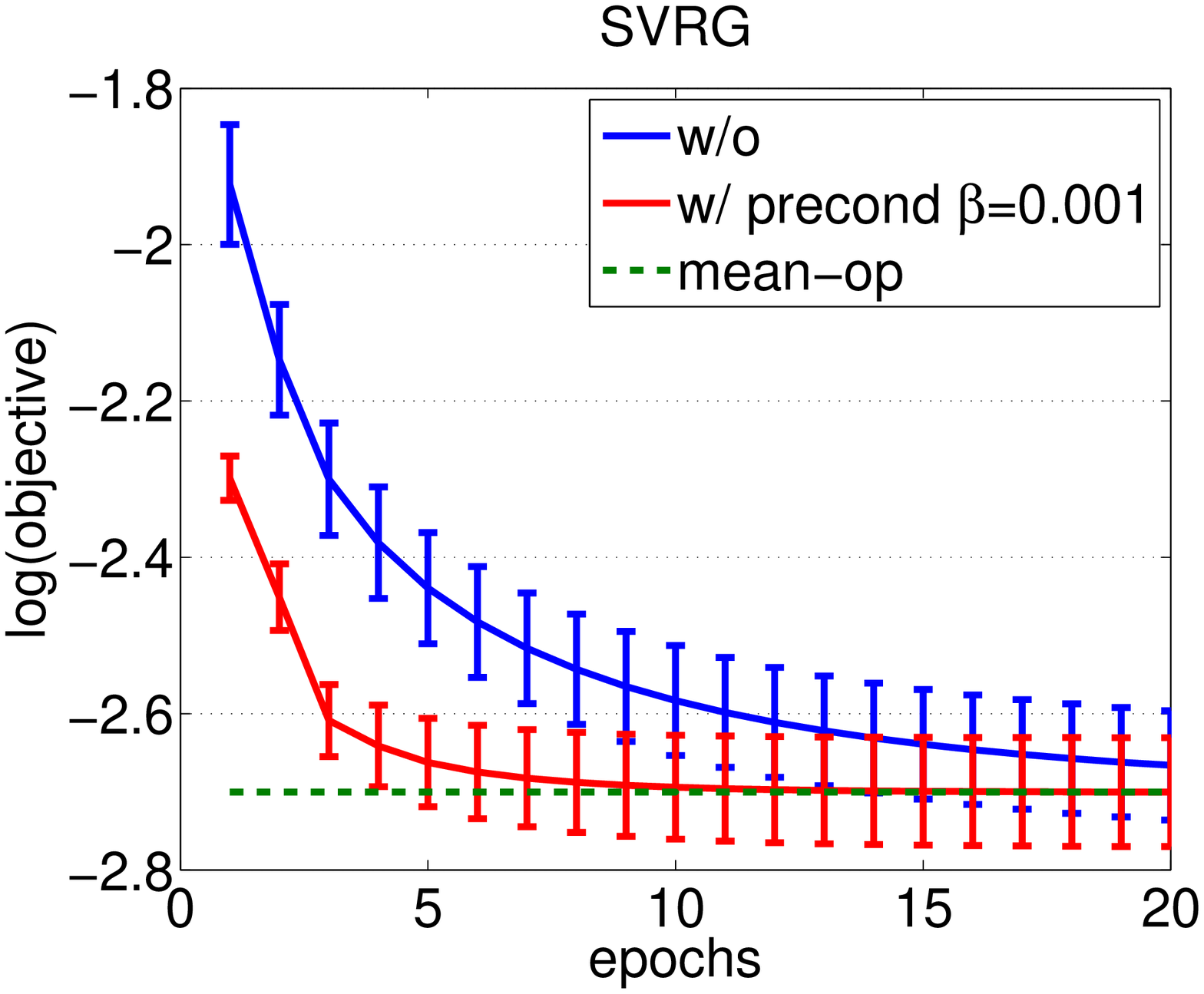}}} \hspace*{0.15in}
\subfigure[poly-$\tau$ ($0.5$), $\lambda=10^{-5}$]{\includegraphics[scale=0.25]{{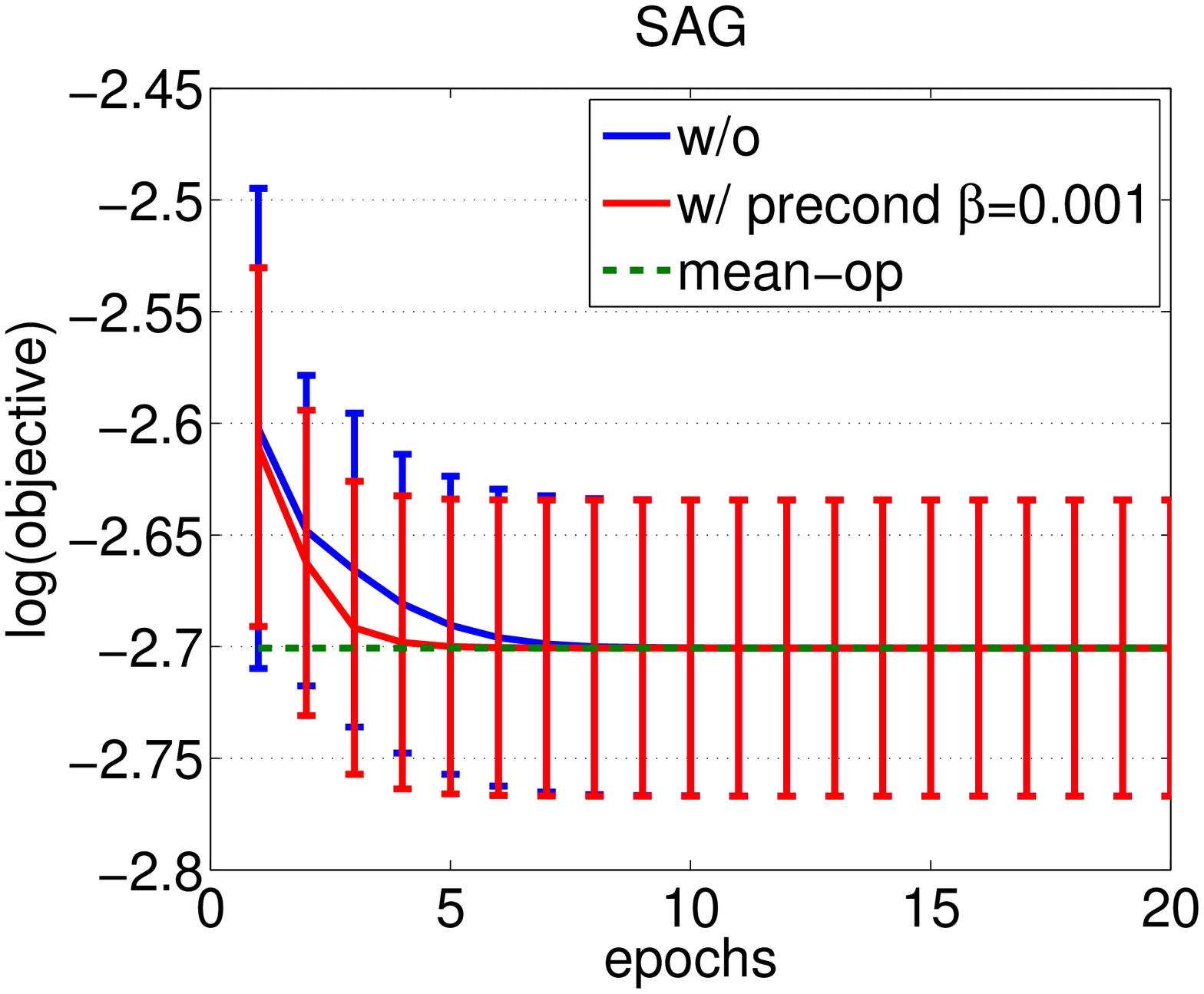}}}
\caption{Convergence of two SGD variants w/ and w/o  data preconditioning for solving the least square problem (a,b) and logistic regression problem   on the synthetic data with the eigenvalues following a polynomial decay. The value of $\lambda$ is set to $10^{-5}$. The condition numbers of the two problems are reduced from $=2727813$ and $681953$ to $c'=1.88$, and $32506$, respectively. 
}\label{fig:convlsreg}
\end{figure}

\begin{figure}
\centering
\subfigure[poly-$\tau (0.5)$, $\beta=0.99$]{\includegraphics[scale=0.25]{{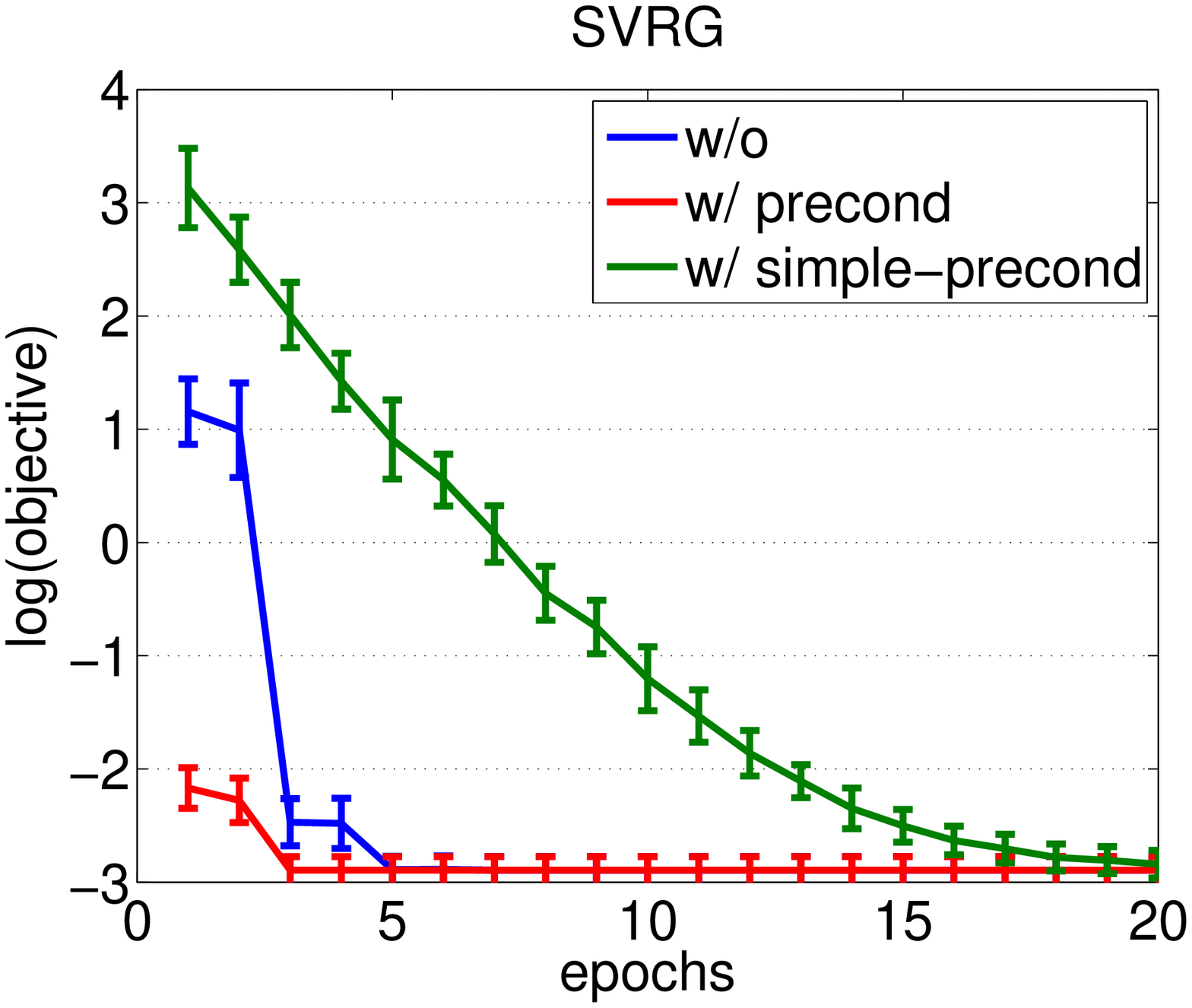}}}\hspace*{0.1in}
\subfigure[poly-$\tau (0.5)$, $\beta=0.99$]{\includegraphics[scale=0.25]{{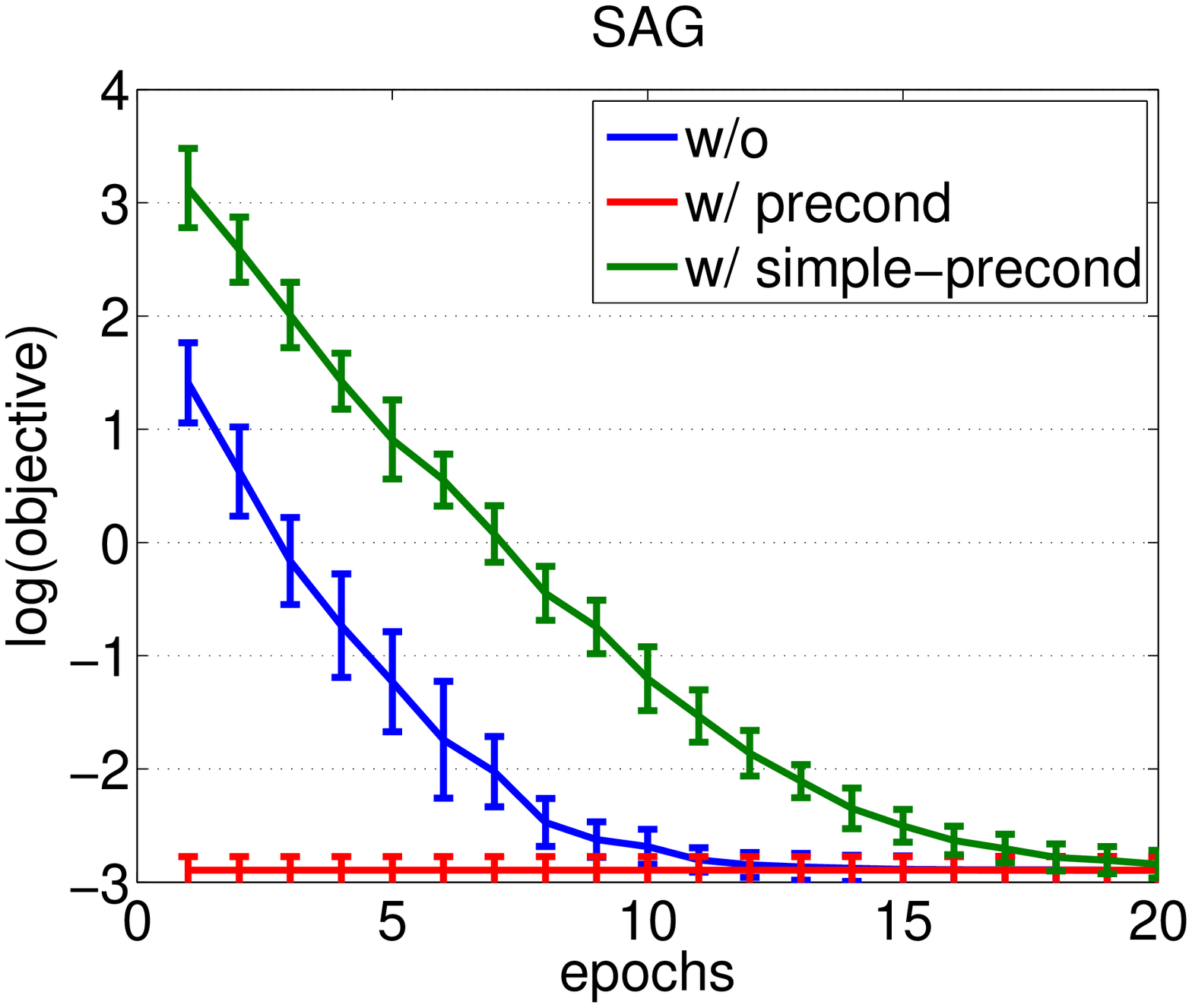}}}

\subfigure[exp-$\tau (0.5)$, $\beta=0.99$]{\includegraphics[scale=0.25]{{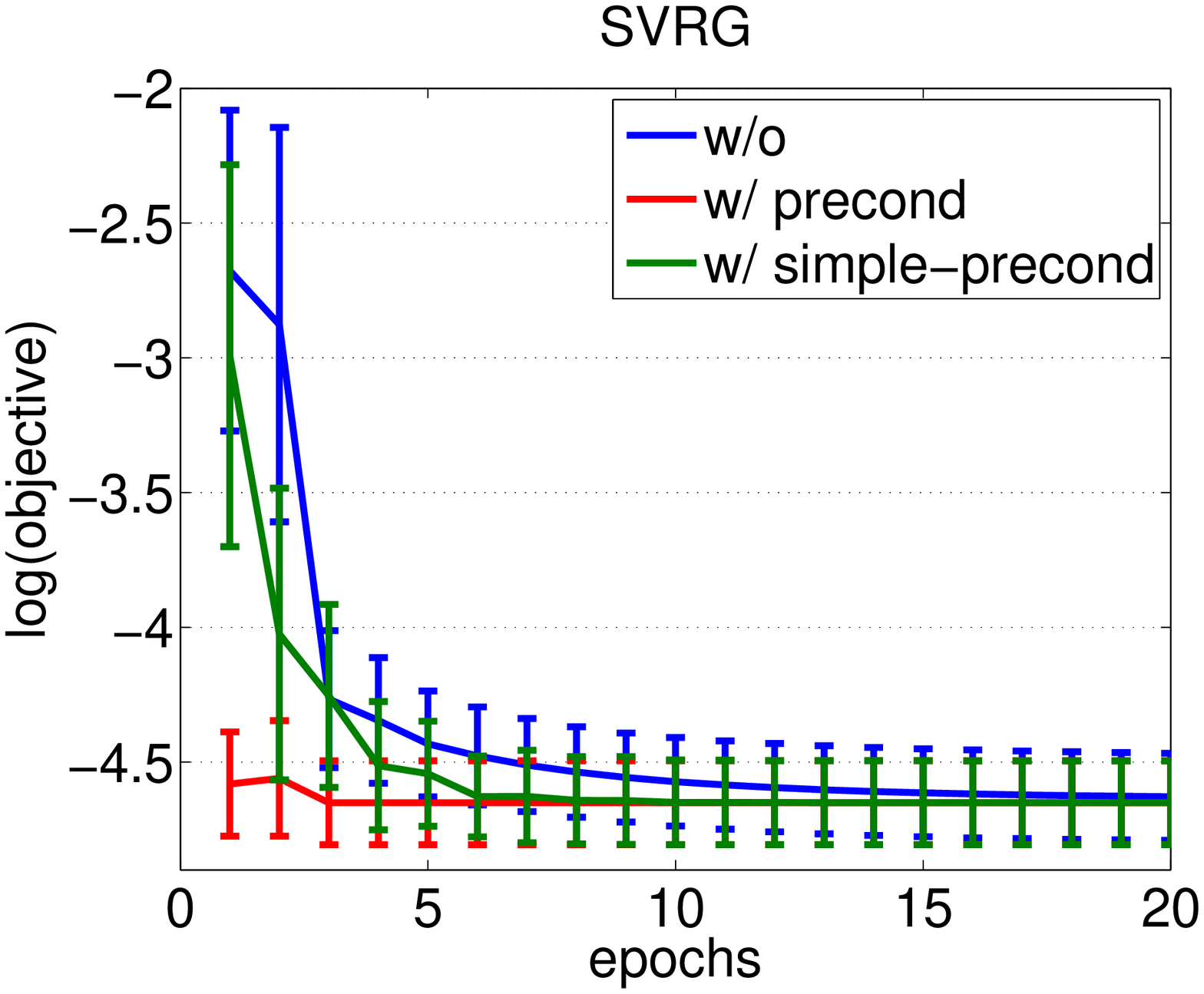}}}
\subfigure[exp-$\tau (0.5)$, $\beta=0.99$]{\includegraphics[scale=0.25]{{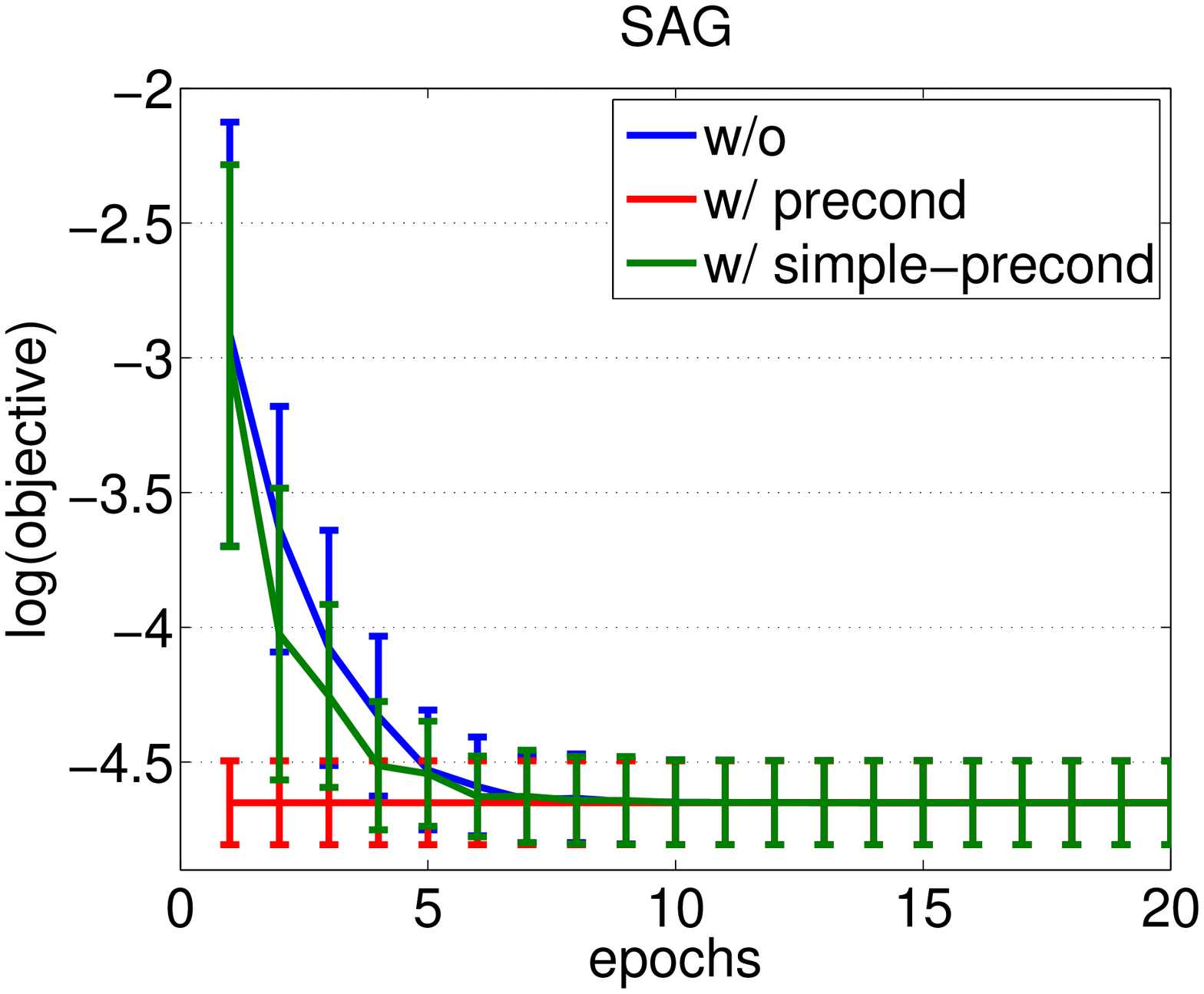}}}
\caption{Comparison of the proposed data preconditioning with the straightforward approach by solving~(\ref{eqn:P1}) (simple-precond) on the synthetic regression data generated with different decay of eigen-values.}\label{fig:cp1}
\end{figure}

\begin{figure}
\centering
\subfigure[$d=5000, \beta=0.99$]{\includegraphics[scale=0.25]{{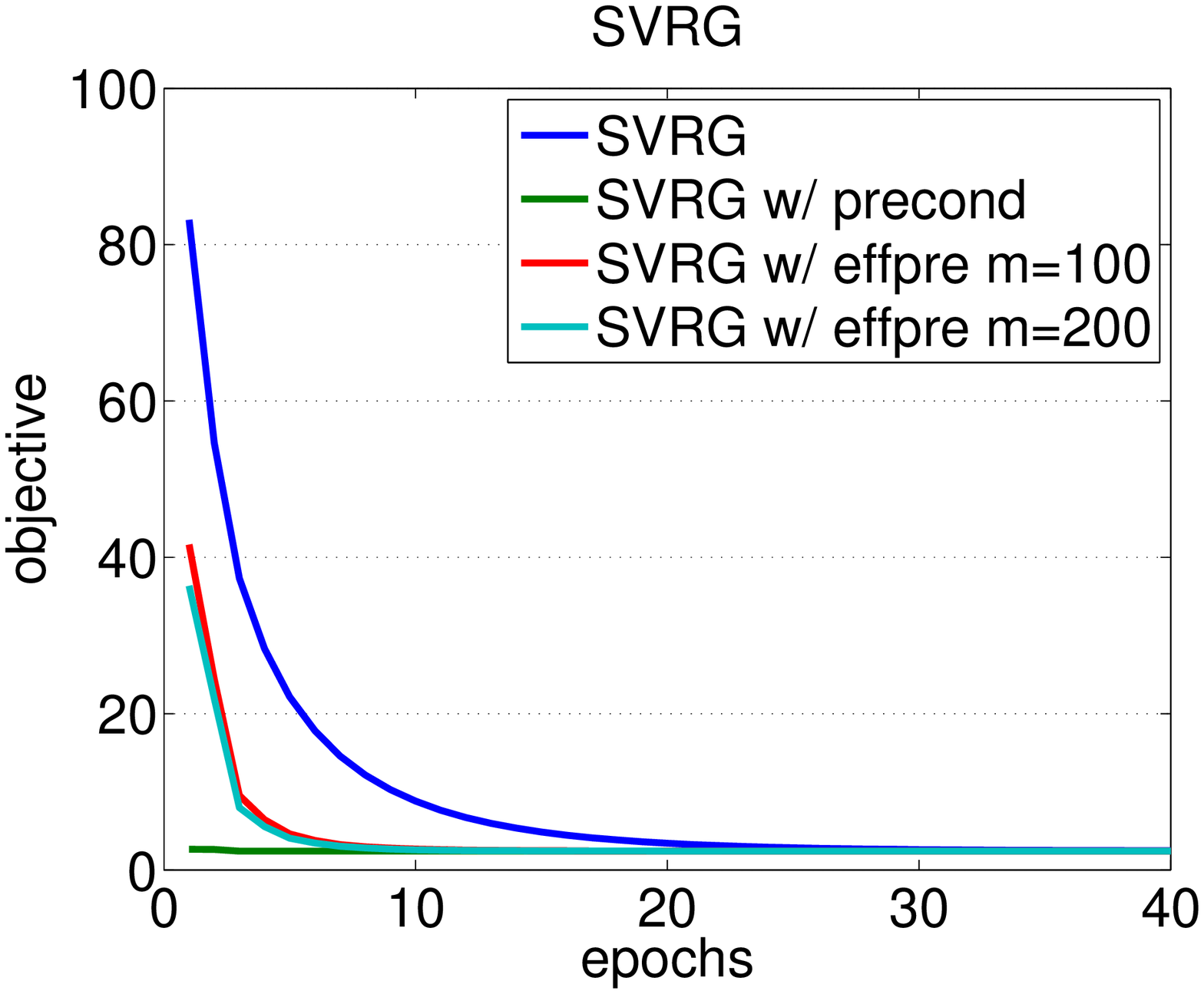}}}
\subfigure[$d=5000, \beta=0.001$]{\includegraphics[scale=0.25]{{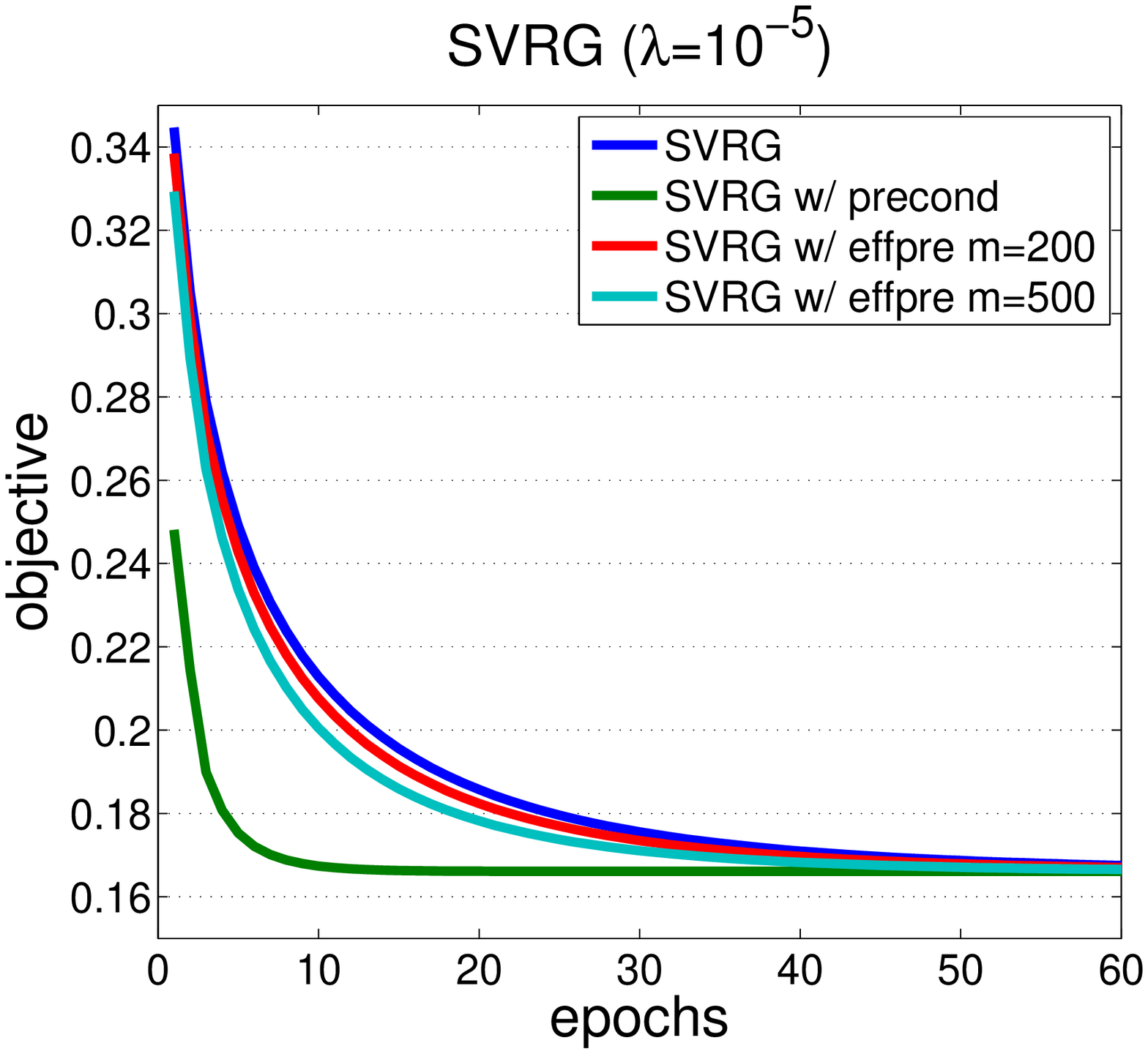}}}

\caption{ Comparison of   convergence of SVRG using full data and sub-sampled data for constructing the preconditioner on the synthetic  data with $d=5000$ features for regression (left) and logistic regression (right).  
}\label{fig:eff}
\end{figure}

Next, we present some experimental  results on convergence. 
In our experiments we focus on two tasks namely least square regression and logistic regression, and we study two variance reduced SGDs namely stochastic average gradient (\textbf{SAG})~\cite{DBLP:journals/corr/SchmidtRB13} and stochastic variance reduced SGD (\textbf{SVRG})~\cite{NIPS2013_4937}.  For SVRG, we set the step size as $0.1/\tilde L$,  where $\tilde L$ is the smoothness parameter of  the individual loss function plus the regularization term in terms of $\w$. The number of iterations for the inner loop in SVRG is set to $2n$ as suggested by the authors.   For SAG, the theorem indicates the step size is less than  $1/(16\tilde L)$ while the authors have reported that using large step sizes like $1/\tilde L$ could yield better performances. Therefore we use $1/\tilde L$ as the step size unless otherwise specified. Note that we are not aiming to optimize the performances by using pre-trained  initializations~\cite{NIPS2013_4937} or by tuning the step sizes. Instead, the initial solution for all algorithms are set to zeros and  the step sizes used in our experiments are either suggested in previous papers or have been observed to perform well in practice.  In all experiments, we compare  the convergence vs the number of epochs.

 We generate synthetic data as described above. 
 For least square regression,  the response variable is generated by $y=\w^{\top}\x + \varepsilon$, where $w_i\sim\mathcal N(0, 100)$ and $\varepsilon\sim\mathcal N(0,0.01)$. For logistic regression, the label is generated by $y=sign(\w^{\top}\x + \varepsilon)$. Figure~\ref{fig:convlsreg} shows the objective curves for minimizing the two problems by SVRG, SAG w/ and w/o data preconditioning. The results clearly demonstrate data preconditioning can significantly boost the convergence.  
 
 To further justify the proposed theory of data preconditioning, we also compare with the straightforward approach that solves the preconditioned problem in~(\ref{eqn:P1}) with the same data preconditioner. The results are shown  in Figure~\ref{fig:cp1}.  These results verify that using the straightforward data preconditioning may not boost the convergence. 

Finally, we  validate the  performance of the efficient data preconditioning presented in Section~\ref{sec:imp}. We generate a synthetic data as before with $d=5000$ features and with eigenvalues following the poly-$0.5$ decay, and plot the convergence of SVRG for solving least square regression and logistic regression with different preconditioners, including $H^{-1/2}$ and $\widehat H^{-1/2}$ with different values of $m$. The results are shown in Figure~\ref{fig:eff}, which demonstrate that  using a small number $m$  ($m=100$ for regression and $m=500$ for logistic regression) of training samples for constructing the data preconditioner  is sufficient to gain substantial boost in the convergence. 

 \begin{table}[t]
\caption{the statistics of real data sets}\label{tab:1}
\begin{tabular}{llll}
\toprule
data set&$n$&$d$&task\\
\midrule
covtype&581012&54&classification\\
MSD&463715&90&regression\\
CIFAR-10&10000&1024&classification\\
E2006-tfidf&19395&150350&regression\\
\bottomrule
\end{tabular}
\label{tab:data}
\end{table}

\subsection{Real  Data} Next, we present some experimental results on real data sets. We choose four data sets, the million songs data (MSD)~\cite{conf/ismir/Bertin-MahieuxEWL11} and the E2006-tfidf data~\footnote{\url{http://www.csie.ntu.edu.tw/~cjlin/libsvmtools/datasets/regression.html}}~\cite{Kogan:2009:PRF:1620754.1620794}  for regression, and the CIFAR-10 data~\cite{citeulike:7491128} and the covtype data~\cite{Blackard:1998:CNN:928509}  for classification.  The task on covtype is to predict the forest cover type from cartographic variables.  The task on MSD is to predict the year of a song based on the audio features. Following the previous work, we map the target variable of year from $1922\sim 2011$ into $[0,1]$. The task on CIFAR-10 is to predict the object in $32\times 32$ RGB images. Following~\cite{citeulike:7491128}, we use the mean centered pixel values as the input. We construct a binary classification problem to classify dogs from cats with a total of $10000$ images. . The task on E2006-tfidf is to predict the volatility of stock returns based on the SEC-mandated financial text report, represented by tf-idf.   The size of these data sets are summarized in Table~\ref{tab:1}. We minimize regularized least square loss and regularized logistic loss for regression and classification, respectively. 

\begin{figure}[t]
\centering

\subfigure{\label{fig:r1}\includegraphics[scale=0.25]{{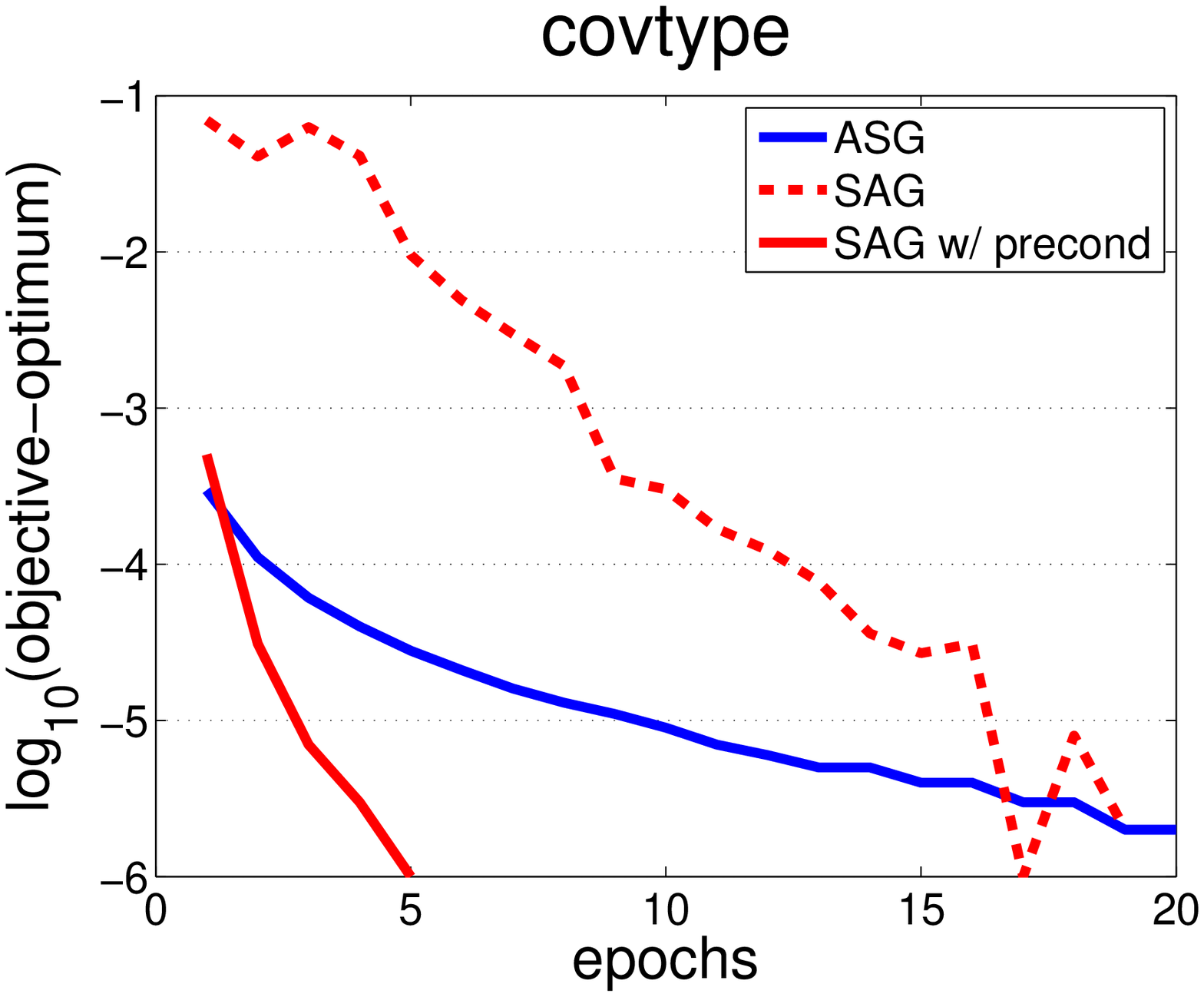}}}
\subfigure{\label{fig:r1}\includegraphics[scale=0.25]{{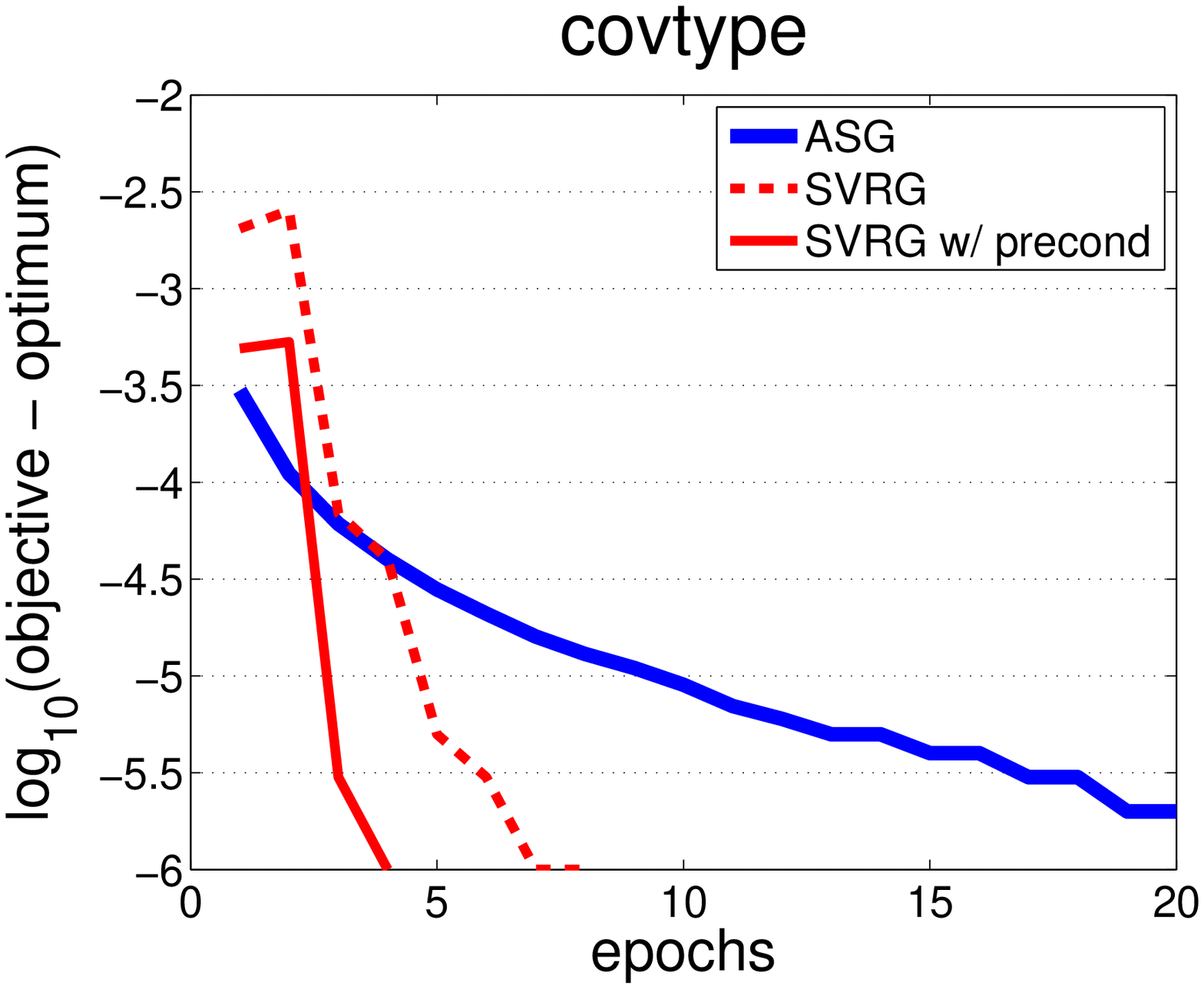}}}
\vspace*{-0.1in}
\caption{comparison of convergence  on covtype.  The value of $\lambda$ is set to $1/n$,  and the value of $\beta$ is $0.01$ for classification.  
}\label{fig:5}
\end{figure}

 \begin{figure}[t]
\centering

\subfigure{\label{fig:r1}\includegraphics[scale=0.25]{{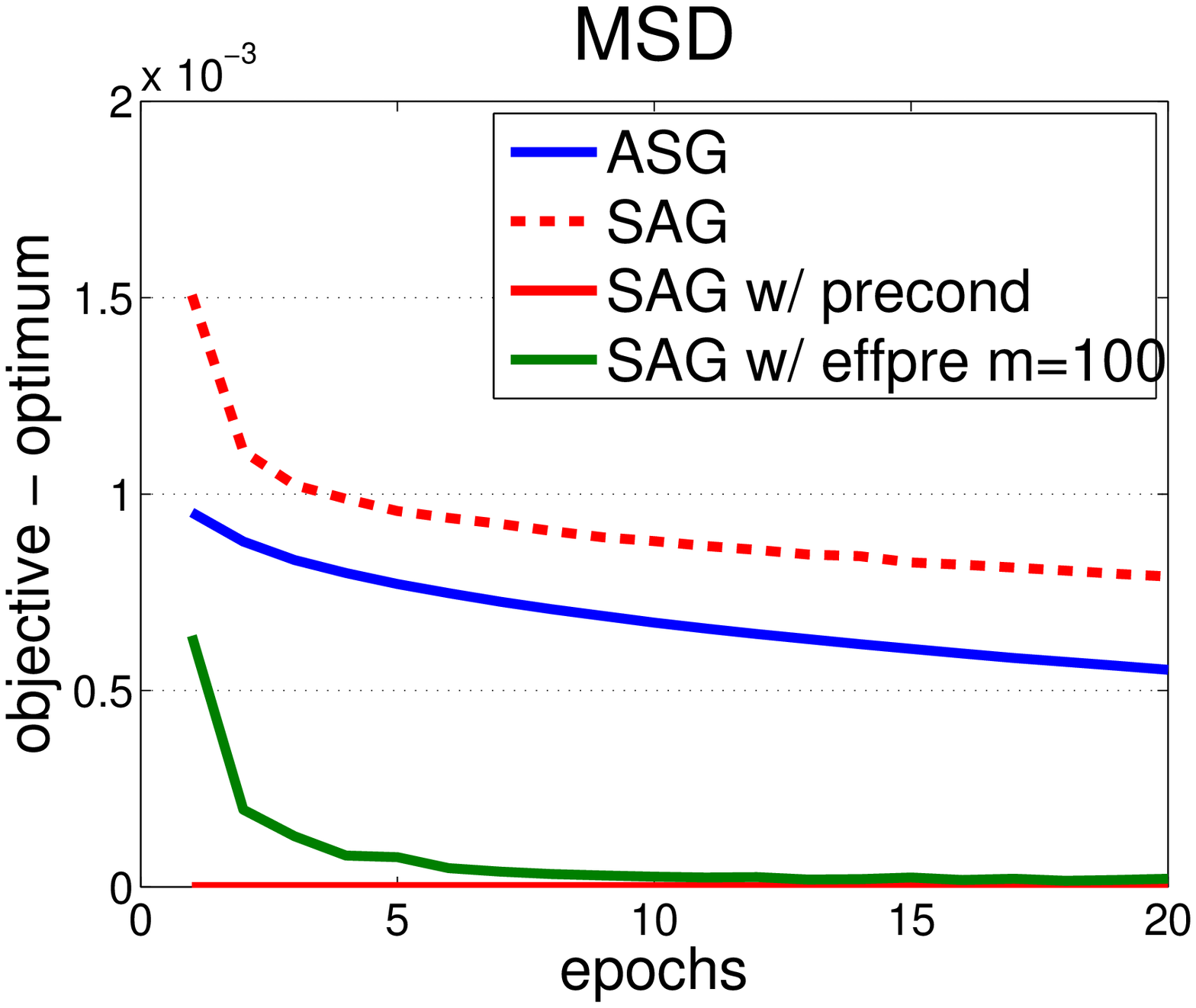}}}
\subfigure{\label{fig:r1}\includegraphics[scale=0.25]{{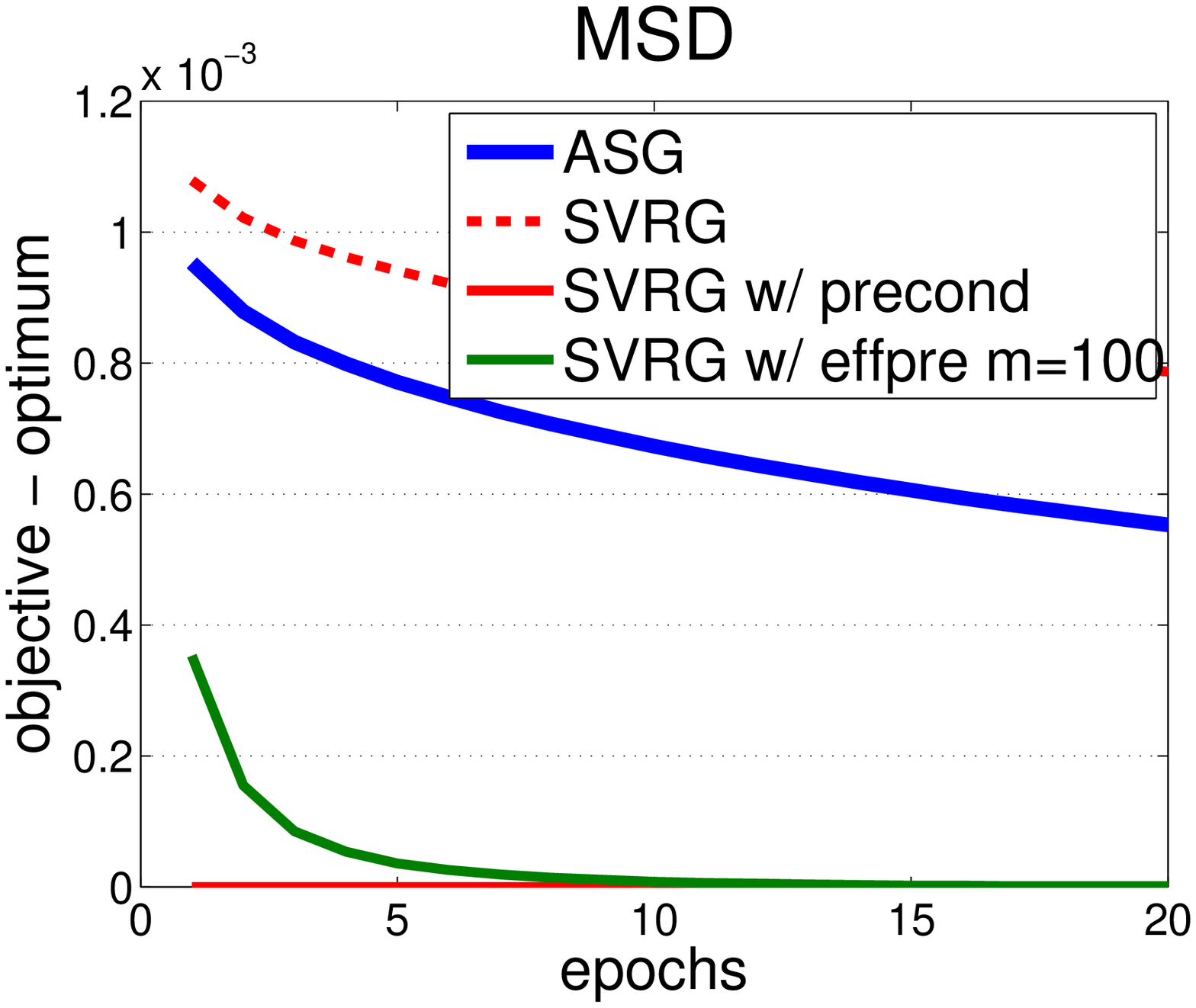}}}
\vspace*{-0.1in}
\caption{comparison of convergence  on MSD. The value of $\lambda$ is set to $2\times 10^{-6}$ MSD,  and the value of $\beta$ is $0.99$ for regression.  
}\label{fig:6}
\end{figure}

 \begin{figure}[t]
\centering

\subfigure{\includegraphics[scale=0.25]{{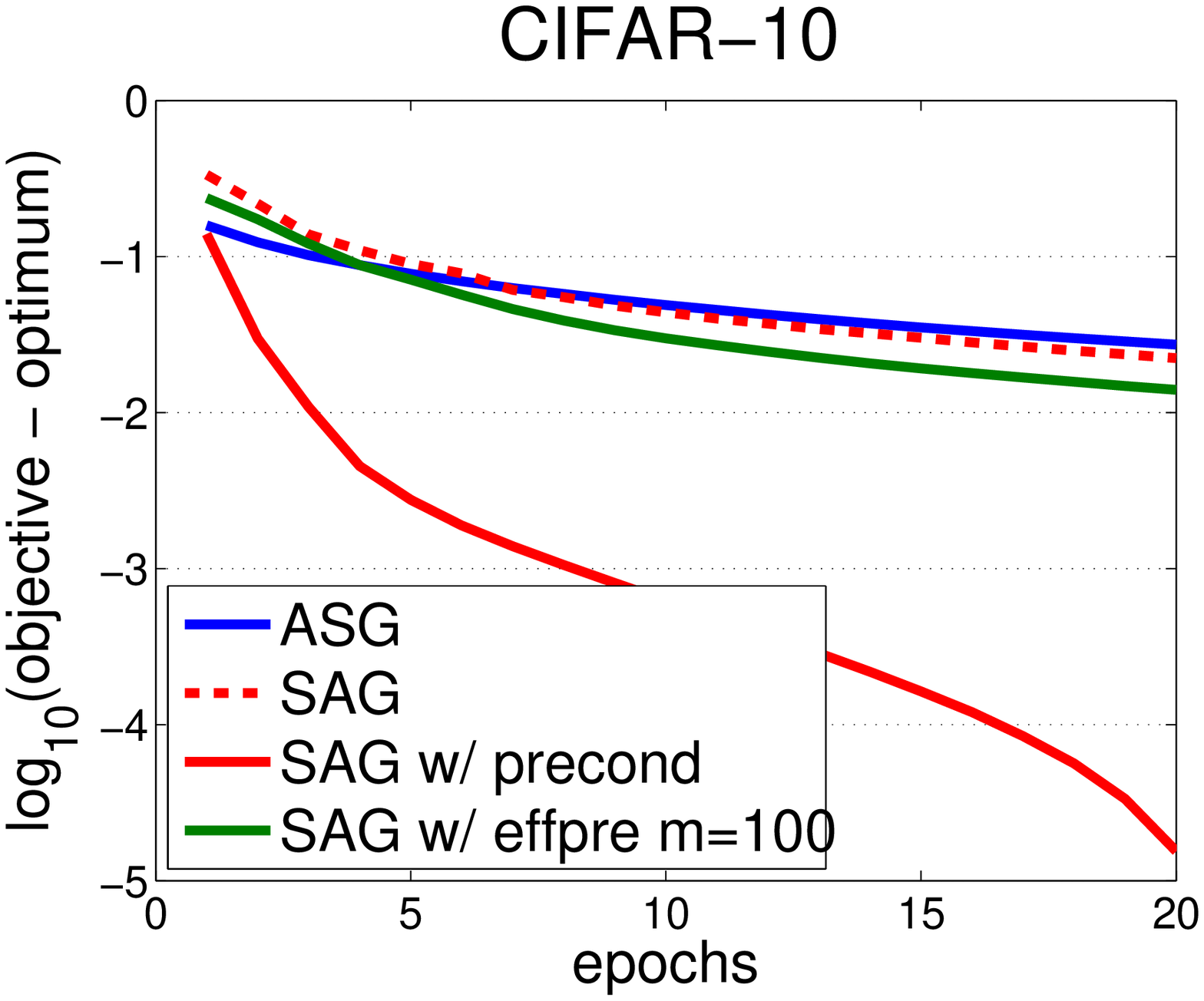}}}
\subfigure{\includegraphics[scale=0.25]{{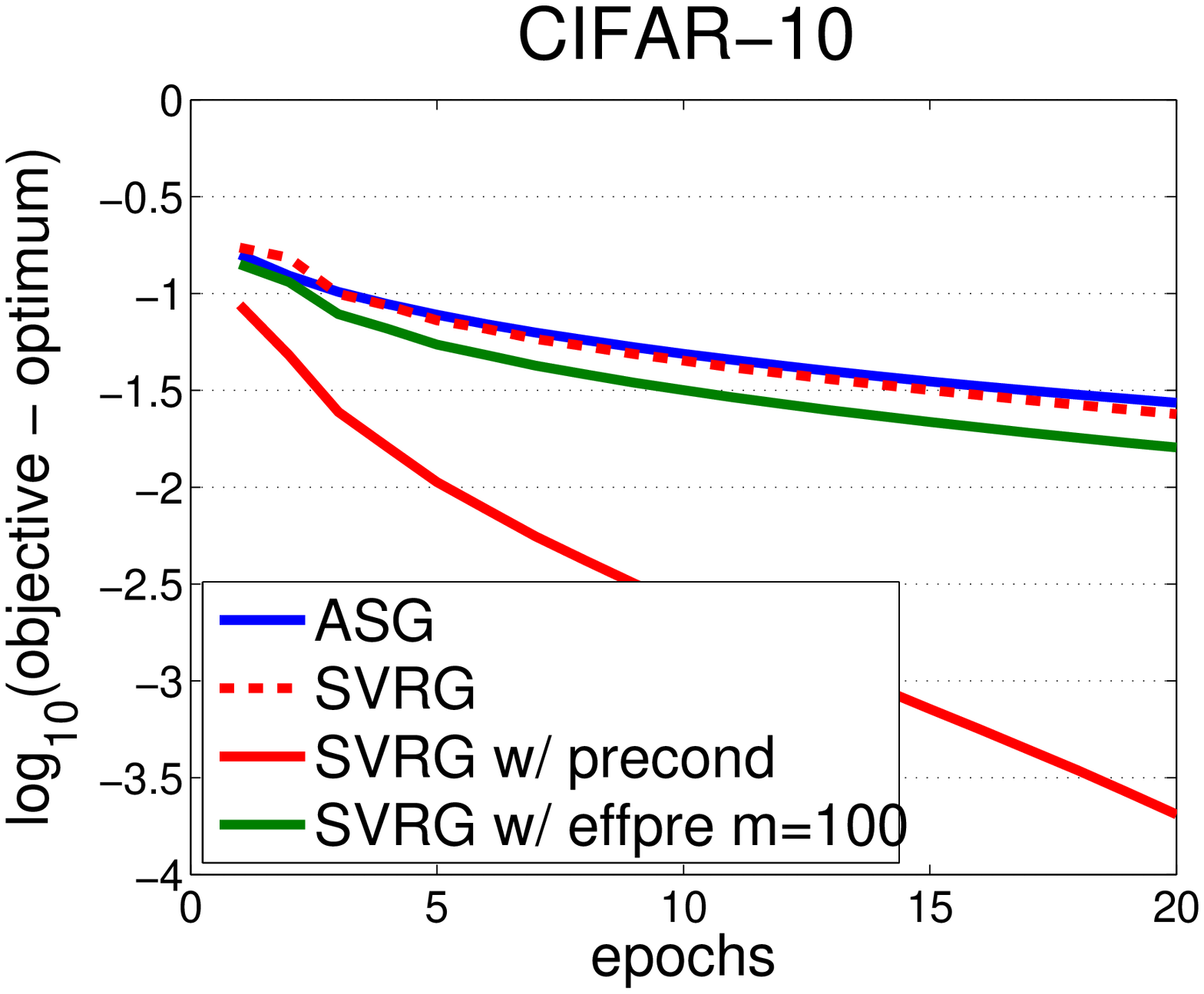}}}\hspace*{-0.14in}
\vspace*{-0.1in}
\caption{comparison of convergence on CIFAR-10. The value of $\lambda$ is set to   $10^{-5}$ for CIFAR-10,  and the value of $\beta$ is $0.01$ for classification.  
}\label{fig:7}
\end{figure}

 The experiment results and the setup are shown in Figures~\ref{fig:5} $\sim$ \ref{fig:8}, in which we also report the convergence of Bach and Moulines' \textbf{ASG} algorithm~\cite{DBLP:conf/nips/BachM13} on the original problem with a step size $c/R^2$, where $c$ is tuned in a range from 1 to $10$. The step size for both SAG and SVRG is set to $1/\tilde L$.  In all figures, we plot the relative objective values~\footnote{the distance of the objective values to the optimal value.} either in log-scale or standard scale versus the epochs. For obtaining the optimal objective value, we run the fastest algorithm sufficiently long until the objective value keeps the same or is within $10^{-8}$ precision.   On MSD and CIFAR-10, the convergence curves of optimizing the preconditioned data problem  using both the full data preconditioning and the sampling based data preconditioning are plotted.  On covtype,  we only plot the convergence curve for optimization using the full data preconditioning, which is efficient enough. On E2006-tfidf, we only conduct optimization using the sampling based data preconditioning because the dimensionality is very large which renders the full data preconditioning very expensive. 
 These results again demonstrate that  the data preconditioning could yield significant speed-up in convergence, and the sampling based data preconditioning could be useful for high-dimensional problems.

 Finally, we report some results on the running time. The computational overhead of the data preconditioning on the four data sets~\footnote{The running time on MSD, CIFAR-10, and E2006-tfidf is  for the sampling based data preconditioning and that on covtype is for the full data preconditioning.} running on Intel Xeon 3.30GHZ CPU is shown in Table~\ref{tab:2}. These computational overhead is marginal or comparable to running time per-epoch. Since the convergence on the preconditioned problem is  faster than that on the original problem by tens of epochs, therefore the training  on the preconditioned problem is  more efficient than that on the original problem. As an example, we plot the relative objective value versus the running time on E2006-tfidf dataset in Figure~\ref{fig:9}, where for SAG/SVRG with efficient preconditioning we count the preconditioning time at the beginning.

\begin{figure}[t]
\centering

\subfigure{\includegraphics[scale=0.25]{{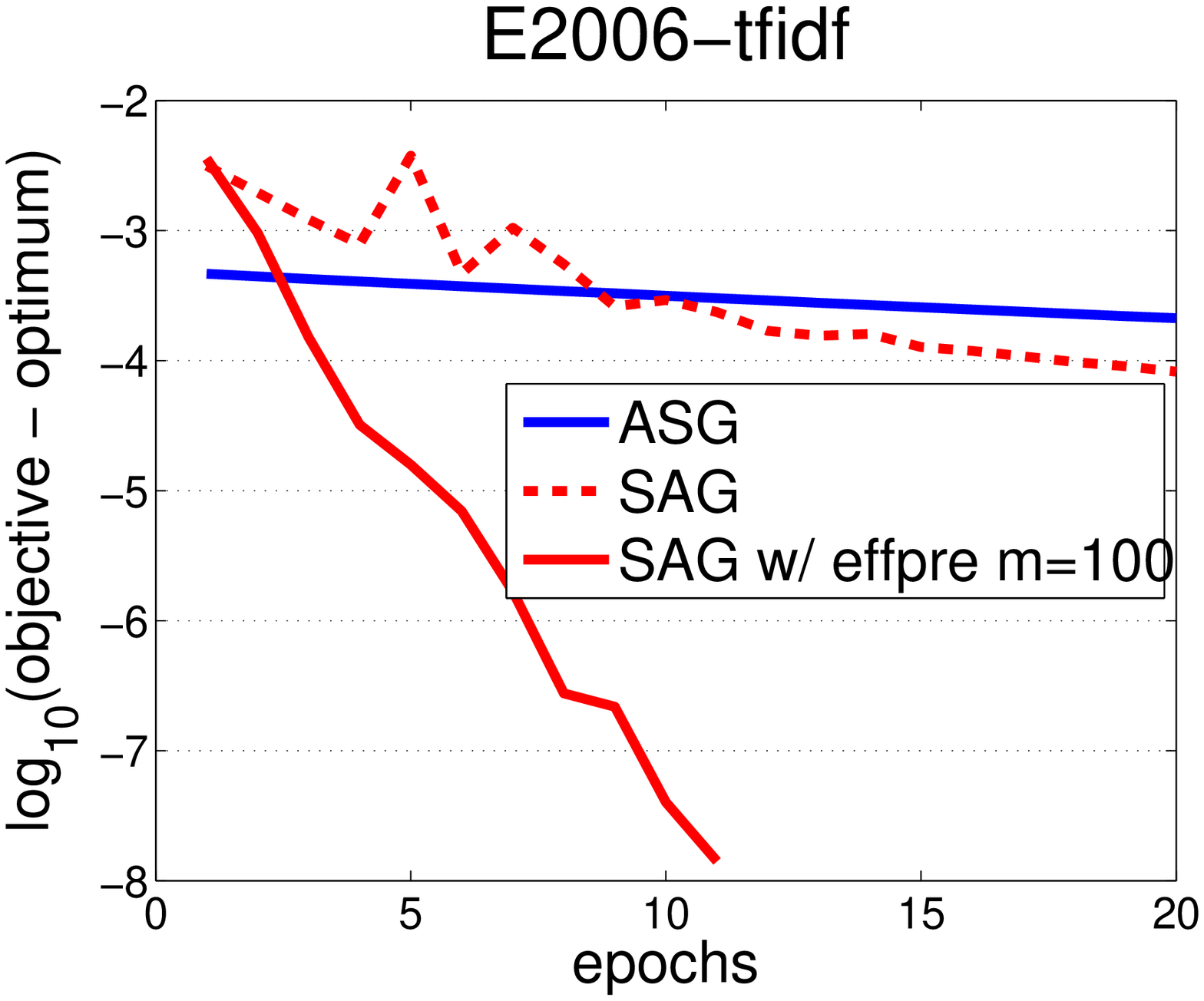}}}
\subfigure{\includegraphics[scale=0.25]{{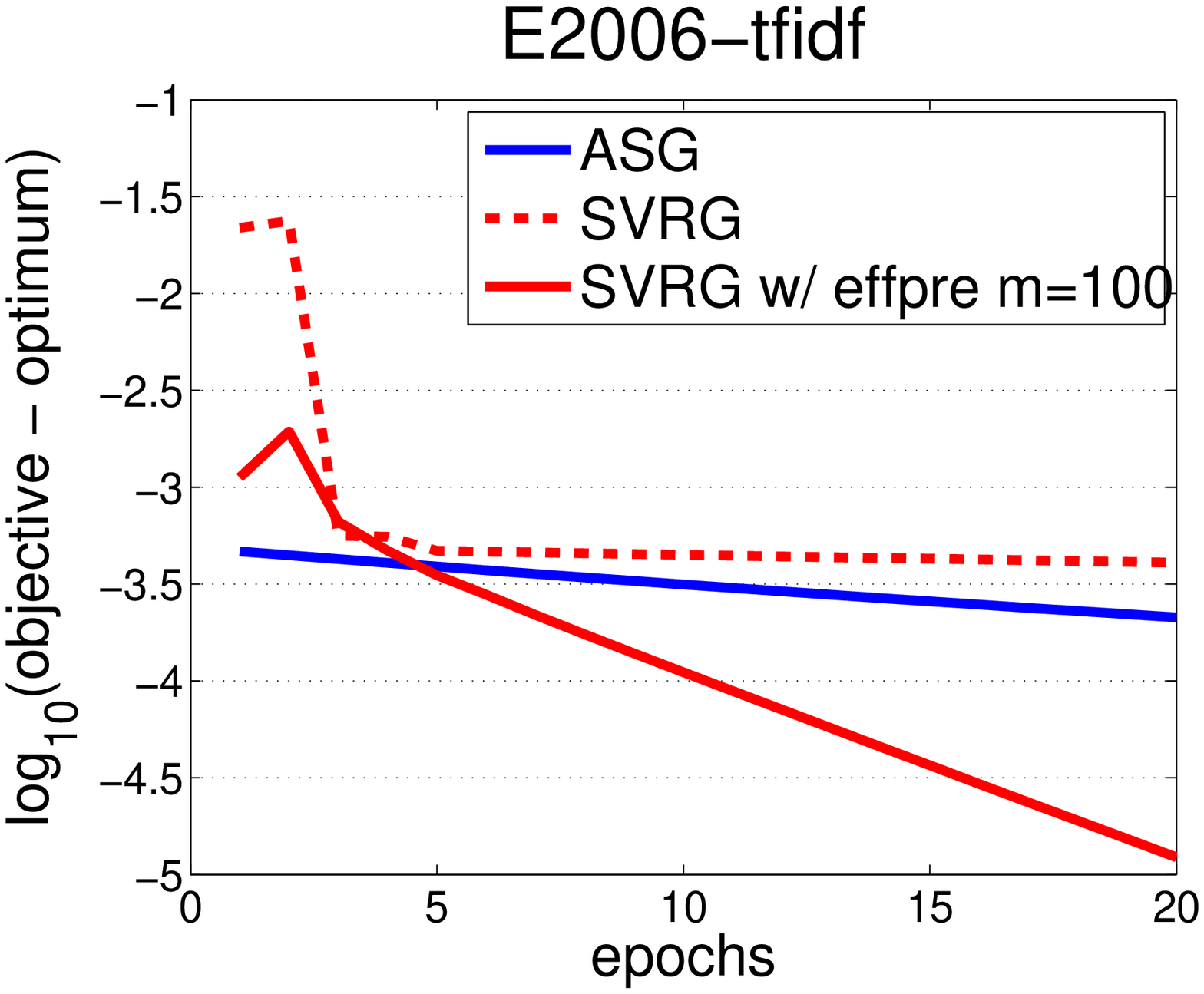}}}
\caption{comparison of convergence  on E2006-tfidf. The value of $\lambda$ is set to $1/n$,  and the value of $\beta$ is $0.99$ for regression.  
}\label{fig:8}
\end{figure}

\section{Conclusions}
We have presented a theory of data preconditioning for boosting the convergence of first-order optimization methods for the regularized loss minimization. We characterized the conditions on the loss function and the data under which the condition number of the regularized loss minimization problem can be reduced and thus the convergence can be improved. We also presented an efficient sampling based data preconditioning which could be useful for high dimensional data, and  analyzed the condition number. Our experimental results  validate our theory and demonstrate the potential advantage of the data preconditioning for solving ill-conditioned regularized loss minimization problems. 

 \begin{table}[t]
\caption{running time of preconditioning (p-time)}\label{tab:2}
\begin{tabular}{lllll}
\toprule
&covtype &MSD&CIFAR-10&E2006-tfidf\\
\midrule
p-time&1.18s&0.30s&0.56s&12s\\
\bottomrule
\end{tabular}
\label{tab:data}
\end{table}

\begin{figure}[t]
\centering

\subfigure{\includegraphics[scale=0.25]{{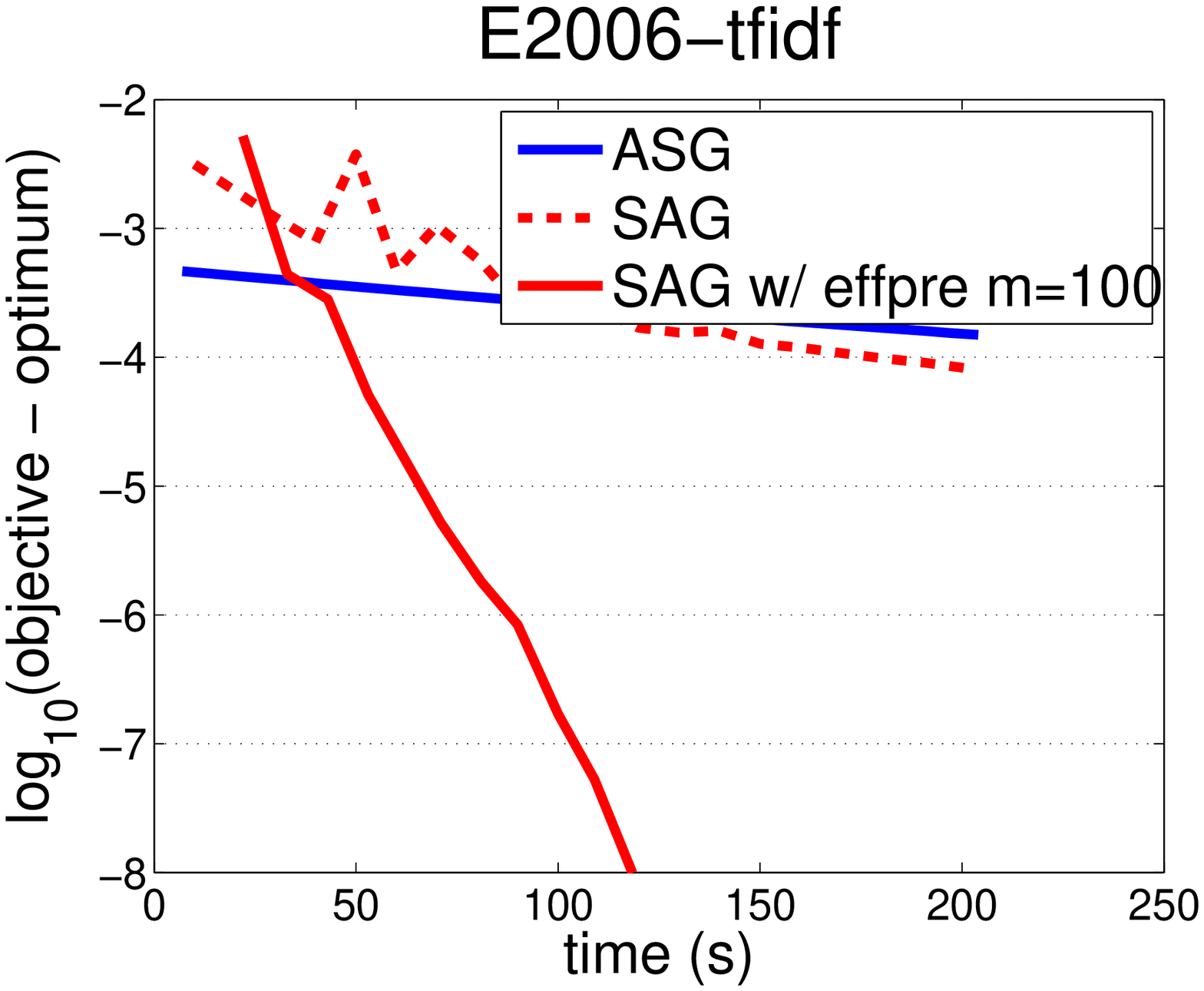}}}
\subfigure{\includegraphics[scale=0.25]{{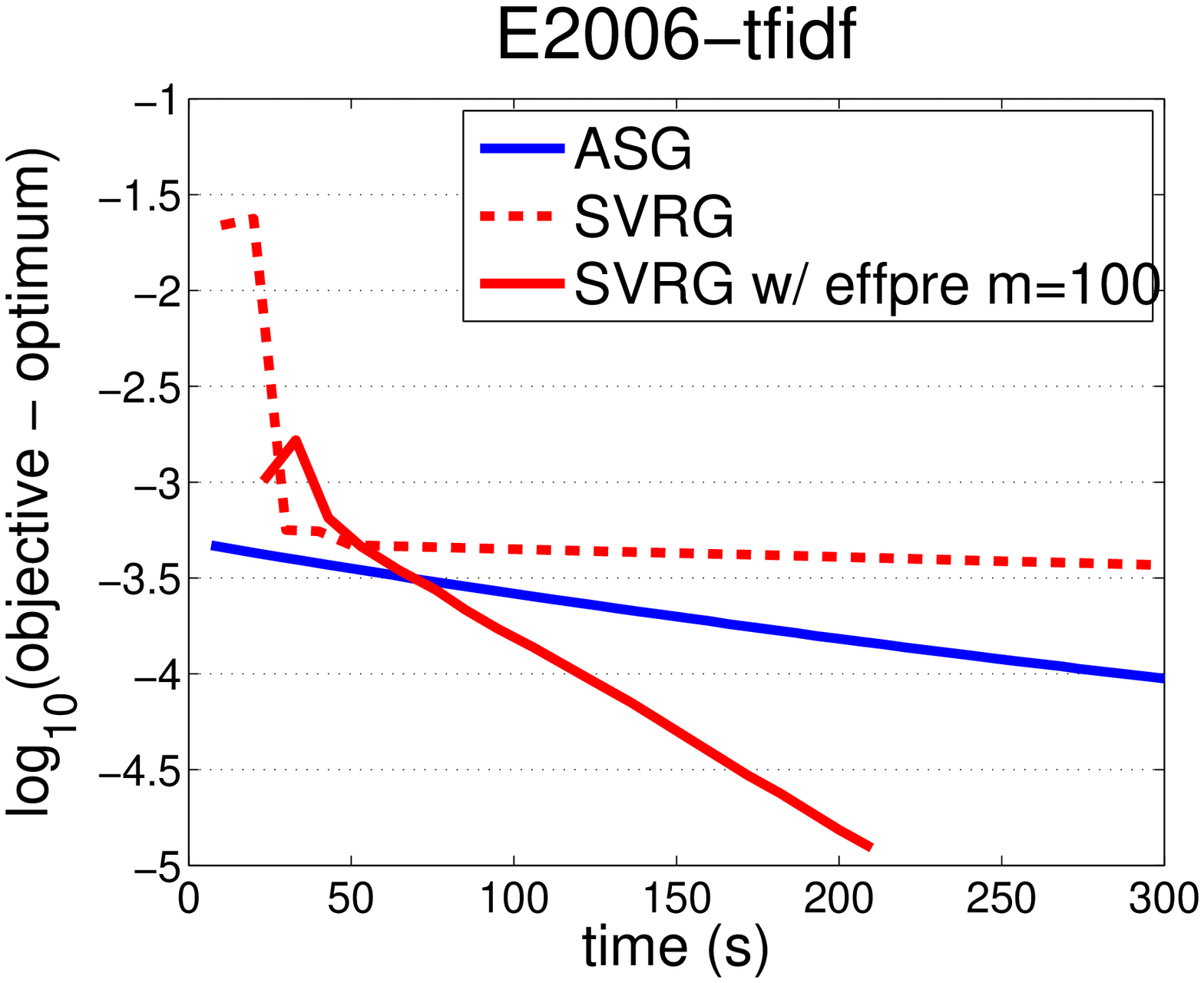}}}
\caption{comparison of convergence versus running time on E2006-tfidf. The value of $\lambda$ is set to $1/n$,  and the value of $\beta$ is $0.99$ for regression.  
}\label{fig:9}
\end{figure}

\section*{Acknowledgements}
The authors would like to thank the anonymous reviewers for their helpful and insightful
comments. T. Yang was supported in part by NSF (IIS-1463988) and NSF (IIS-1545995).


\section*{Appendix A: Proof of Proposition~\ref{prop:1}}
We first prove for the case of polynomial decay $\sigma^2_i = i^{-2\tau}, \tau\geq1/2$.
\begin{align*}
\sum_{i=1}^d\frac{\sigma_i^2}{\sigma_i^2  + \rho} &= \sum_{i=1}^d \frac{1}{ 1+ i^{2\tau}\rho}\leq \int^d_0 \frac{1}{1+t^{2\tau}\rho}dt\\
&=\int^{\rho d^{2\tau}}_0 \frac{1}{1+s}\rho^{-1/(2\tau)}s^{1/(2\tau) - 1}\frac{1}{2\tau}d s\text{ (with the change of variable $s=\rho t^{2\tau}$)}\\
&\leq \int^{\infty}_0 \frac{1}{1+s}\rho^{-1/(2\tau)}s^{1/(2\tau) - 1}\frac{1}{2\tau}d s\\
& = O(\rho^{-1/(2\tau)}) \text{  (since the integral is finite)}
\end{align*}
For the exponential decay $\sigma_i^2 = e^{-\tau i}$, we have
\begin{align*}
\sum_{i=1}^d\frac{\sigma_i^2}{\sigma_i^2 + \rho} &= \sum_{i=1}^d\frac{e^{-\tau i}}{ e^{-\tau i}+\rho}\leq \int^d_0 \frac{e^{-\tau t}}{e^{-\tau t}+\rho}dt\\
&= \frac{1}{\tau}\int^{1}_{e^{-\tau d}} \frac{s}{s + \rho}ds \text{ (with the change of variable $s = e^{-\tau t}$)}\\
&\leq \frac{1}{\tau}\int^{1}_{0} \frac{s}{s + \rho}ds\leq \frac{1}{\tau}\int^{1}_{0} \frac{1}{s + \rho}ds\\
& = \frac{1}{\tau}[\log(1+\rho) - \log(\rho)] = O\left(\log\left(\frac{1}{\rho}\right)\right)
\end{align*}

\section*{Appendix B: Proof of Theorem~\ref{thm:key} }
\begin{proof}
Let us re-define $H=\hat\rho I + C$.  We first show that the upper bound of the preconditioned data norm using $\Hh^{-1}$ is only scaled-up by a constant factor (e.g., $2$) compared to that using $H^{-1}$. We can first bound $\x_i^{\top}\Hh^{-1}\x_i$ by $\x_i^{\top}H^{-1}\x_i$
\begin{align*}
\x_i^{\top}\Hh^{-1}\x_i &  = \x_i^{\top}H^{-1/2}\left(H^{1/2}\Hh^{-1}H^{1/2}\right)H^{-1/2}\x_i\\
&\leq \lambda_{\max}\left(H^{1/2}\Hh^{-1}H^{1/2}\right)\x_i^{\top}H^{-1}\x_i, \quad i=1,\ldots, n.
\end{align*}
So the crux of bounding $\x_i^{\top}\Hh^{-1}\x_i$ is to bound $\lambda_{\max}\left(H^{1/2}\Hh^{-1}H^{1/2}\right)$, i.e., the largest eigenvalue of $H^{1/2}\Hh^{-1}H^{1/2}$.  To proceed the proof, we need the following Lemma.  
\begin{lemma}\cite{journals/corr/abs-1011-1595}
Let $\X$ be a finite set of PSD matrices with dimension $k$, and suppose that
\[
    \max_{X \in \X} \lambda_{\max}(X) \leq B.
\]
Sample $\{X_1, \ldots, X_{\ell}\}$ uniformly at random from $\X$ without replacement. Compute
\[
    \mu_{\max} = \ell \lambda_{\max}(\E[X_1]), \quad     \mu_{\min} = \ell \lambda_{\min}(\E[X_1])
\]
Then
\begin{align*}
& \Pr\left\{\lambda_{\max}\left(\bar X\right) \geq (1 + \delta) \mu_{\max} \right\} \leq k \left[ \frac{e^{\delta}}{(1 + \delta)^{1 + \delta}}\right]^{\frac{\mu_{\max}}{B}} \\
& \Pr\left\{\lambda_{\min}\left(\bar X\right) \leq (1 - \delta) \mu_{\max} \right\} \leq k \left[ \frac{e^{-\delta}}{(1 - \delta)^{1 - \delta}}\right]^{\frac{\mu_{\max}}{B}}
\end{align*}
where $\bar X= \sum_{i=1}^lX_i$. 
\end{lemma}
Let us define  $S= \Sigma^2 + \hat\rho I$  and
\[
\X = \left\{X_i = H^{-1/2}\left(\x_i\x_i^{\top} + \hat\rho I\right)H^{-1/2}, i=1, \ldots, n \right\}
\]
First we show that 
 $$\lambda_{\max}(X_i) \leq \mu(\hat\rho) \gamma(C, \hat\rho) + 1.$$ Since
\[
\mu_{\max} = m \lambda_{\max}(\E_i[X_i]) = m
\]
This can be proved by noting that 
\begin{align*}
\lambda_{\max}(H^{-1/2}\hat\rho IH^{-1/2})&= \max_i\frac{\hat\rho}{\hat\rho + \sigma_i^2}\leq 1\\
\lambda_{\max}(H^{-1/2}\x_i\x_i^{\top}H^{-1/2})&\leq \x_i^{\top}H^{-1}\x_i\leq \mu(\hat\rho) \gamma(C, \hat\rho) 
\end{align*}
where the second inequality is due to Lemma~\ref{lem:2} and the new definition of $H$. 
By applying the above Lemma and noting that $\bar X =\frac{1}{m} \sum_{i=1}^mX_i = H^{-1/2}\Hh H^{-1/2}$,   we have
\begin{align*}
&\Pr\left\{\lambda_{\min}\left(H^{-1/2}\Hh H^{-1/2}\right) \leq 1 - \delta \right\}\\
& \leq d\exp\left( - \frac{m}{\mu(\hat\rho) \gamma(C,\hat\rho) + 1}\left[(1 - \delta)\log(1 - \delta) + \delta\right]\right)
\end{align*}
Using the fact that
\[
(1 - \delta)\log(1 - \delta) \geq - \delta + \frac{\delta^2}{2}
\]
and  by setting $m = 2(\mu(\hat\rho) \gamma(C, \hat\rho) + 1) (\log d + t)/\delta^2$, we have with a probability $1-e^{-t}$, 
\[
\lambda_{\min}\left(H^{-1/2}\Hh H^{-1/2}\right)\geq 1-\delta
\]
As a result, we have with a probability $1 - e^{-t}$, 
\begin{align*}
\lambda_{\max}\left(H^{1/2}\Hh^{-1}H^{1/2}\right)&\leq \frac{1}{\lambda_{\min}\left(H^{-1/2}\Hh H^{-1/2}\right)}\\
&\leq \frac{1}{1-\delta}\leq 1+2\delta, \quad\forall \delta\leq 1/2. 
\end{align*}
Therefore, we have with a probability $1 - e^{-t}$ for any $\delta\leq 1/2$, 
\[
\x_i^{\top}\Hh^{-1}\x_i \leq (1+2\delta)\mu(\hat\rho)\gamma(C,\hat\rho), \quad i=1,\ldots, n
\]
\end{proof}

\bibliography{icml14}
\bibliographystyle{abbrv}

\end{document}